\documentclass[10pt]{amsart}
\usepackage[]{amsmath, amsthm, amsfonts}
\usepackage{graphicx}
\usepackage[all]{xy}
\usepackage{xcolor}
\usepackage{tikz-cd}
\usepackage{adjustbox}
\usepackage{hyperref}
\usepackage[
backend=biber,
style=alphabetic
]{biblatex}

\usepackage[left=1in,right=1in,top=1in,bottom=1in]{geometry}
\usepackage{cleveref}
\newcommand {\N} {{\mathbb N}}
\newcommand {\C} {{\mathbb C}}

\newcommand {\Z} {{\mathbb Z}}
\newcommand {\Q} {{\mathbb Q}}
\newcommand {\cF} {{\mathcal F}}
\newcommand {\tcF} {\tilde{{\mathcal F}}}

\newcommand {\cH} {{\mathcal H}}
\newcommand {\cG} {{\mathcal G}}
\newcommand {\cL} {{\mathcal L}}
\newcommand {\tcL} {\tilde{{\mathcal L}}}

\newcommand {\cM} {{\mathcal M}}
\newcommand {\cV} {{\mathcal V}}

\newcommand {\cO} {{\mathcal O}}
\newcommand {\D} {\mathbb{D}}

\newcommand{\tf}{\tilde{f}}
\newcommand{\tX}{\tilde{X}}

\newcommand{\cD}{\mathcal{D}}
\newcommand{\cN}{\mathcal{N}}
\newcommand{\bR}{\textbf{R}}

\newcommand{\tcM}{\tilde{\cM}}

\newcommand{\bH}{\textbf{H}}
\newcommand{\bT}{\begin{tikzcd}}
\newcommand{\eT}{\end{tikzcd}}

\DeclareMathOperator{\im}{im}

\newtheorem{thm}[subsection]{Theorem}
\newtheorem{cor}[subsection]{Corollary}
\newtheorem{lemma}[subsection]{Lemma}
\newtheorem{prop}[subsection]{Proposition}
\newtheorem{defn}[subsection]{Definition}
\newtheorem{rmk}[subsection]{Remark}
\newtheorem{ex}[subsection]{Example}

\newtheorem{set}[subsection]{Setting}
\newtheorem{nota}[subsection]{Notation}

\begin{document}
\author{Scott Hiatt }
\date{\today}
  \address{
 Department of Mathematics\\
  DePauw University\\
  Greencastle, IN 46135\\
  U.S.A.}
  \email{scotthiatt@depauw.edu}

\title{Vanishing of Local Cohomology with Applications to Hodge Theory}

\maketitle

\begin{abstract}
    Let $\mathbf{H} = ((H, F^{\bullet}), L)$ be a polarized variation of Hodge structure on a smooth quasi-projective variety $U.$ By M. Saito's theory of mixed Hodge modules, the variation of Hodge structure $\bold{H}$ can be viewed as a polarized Hodge module $\mathcal{M} \in HM(U).$ Let $X$  be a compactification of $U,$ and $j:U \hookrightarrow X$ is the natural map. In this paper, we use local cohomology with mixed Hodge module theory to study $j_{+}\mathcal{M} \in D^{b}MHM(X).$ In particular, we study the graded pieces of the de Rham complex $Gr^{F}_{p}DR(j_{+}\mathcal{M}) \in D^{b}_{coh}(X),$ and the Hodge structure of $H^{i}(U, L)$ for $i$ in sufficiently low degrees.
\end{abstract}

\section*{Introduction}

Suppose $\mathcal{X}$ is a smooth complex projective variety and $\bH =((H, F^{\bullet}), L)$ is a polarized variation of Hodge structure on $\mathcal{X}$. $\bH$ consists of a holomorphic vector bundle $H$ on $\mathcal{X}$ with a decreasing filtration $F^{\bullet}=F^{\bullet}H$, and an associated local system $L$ over $\Q$ (see Def. \ref{VHS} for more details). It is well known that $H^{i}(\mathcal{X}, L)$ is naturally equipped with a pure $\Q$-Hodge structure. But if $U$ is a smooth complex quasi-projective variety and $\bH$ is a polarized variation of Hodge structure on $U,$ then in general, $H^{i}(U, L)$ no longer has a pure Hodge structure, but a mixed Hodge structure. Suppose $X$ is a compactification of $U$ of dimension $n,$ and $S = X \backslash U$ with dimension $n_{S}$. D. Arapura \cite[Thm.  2]{arapura} showed that if $i \leq n- n_{S} -2$, then $H^{i}(U, \Q)$ is pure of weight $i$. Furthermore, if $i \leq n -n_{S} -2$ and $p +q = i$, then
\begin{equation}
H^{i}(U, \C) \cong  \displaystyle \bigoplus_{p+q = i}H^{q}(U, \Omega^{p}_{U}) \quad \text{and} \quad \dim H^{q}(U, \Omega^{p}_{U}) = \dim H^{p}(U, \Omega^{q}_{U}).
\end{equation}
The proof of this theorem relies on the vanishing of local cohomology. Specifically, Arapura shows that if $\pi: \mathcal{X} \rightarrow X$ is a strong log resolution, where $D = \pi^{-1}(S)_{red}$ is a simple normal crossing divisor, then
\begin{equation}
H^{q}_{D}(\mathcal{X}, \Omega^{p}_{\mathcal{X}}(\log D) \otimes \pi^{\ast}\cV) = 0 \quad \text{provided that $p +q \leq n- n_{S} -1$},
\end{equation}
where $\cV$ is a locally free sheaf on $X$ \cite[Thm. 1]{arapura}. This vanishing theorem is also used to prove a Kodaira-Nakano type vanishing theorem for quasi-projective varieties \cite[Prop. 1]{arapura}. As well as a variant of the  Lefschetz hyperplane theorem for quasi-projective varieties \cite[Prop. 2]{arapura}.

In this paper, we generalize Arapura's vanishing theorem \cite[Thm. 1]{arapura} (see \ref{nota1} for notation).

\begin{thm}\label{thm.1}(Thm. \ref{thm1.5}) Suppose $X$ is a proper complex variety, and  $f: Y \rightarrow X$ is a proper surjective map.  Let $S \subset X$ be any closed subvariety of $X$ of dimension $n_{S}$ and $E = f^{-1}(S)_{red}.$ If $\mathcal{F} \in D^{b}_{coh}(\cO_Y)$ and there exists $q \in \Z$ such that $R^{i}f_{\ast}\mathcal{F} = 0$ for all $i \geq q$, then
\begin{center}
$Ext^{p}_{E}(\mathcal{F} \otimes f^{*}\cV, \omega^{\bullet}_{Y}) = 0$ for $ p \leq -n_{S} -q $,
\end{center}
where $\cV$ is any locally free sheaf on $X.$
\end{thm}

One of the objectives of this paper is to apply the previous theorem with M. Saito's theory of mixed Hodge modules \cite{saito} \cite{saito2} to show Arapura's results hold for any polarized variation of Hodge structure $\bH$ on $U$.  By \cite{saito}, any polarized variation of Hodge structure $\bH$ on $U$ of weight $w -n$ is a Hodge module of weight $w$. If $X$ is a compactification of $U$, then there is also a unique Hodge module $IC^{H}_{X}(L) \in HM_{X}(X,w)$ that extends $\bH,$ whose underlying perverse sheaf is the intersection complex $IC_{X}(L)$ twisted by $L$. By applying Theorem \ref{thm.1}, two of the main results of this paper are the following generalizations of Arapura's theorem \cite[Thm. 2]{arapura}.

\begin{thm}(Thm. \ref{thm3.5})
Suppose $U$ is a smooth irreducible quasi-projective variety of dimension $n$. Let $X$ be a compactification of $U$ and $S = X \backslash U$ with dimension $n_{S}$. If $\cM \in HM_{U}(U,w)$ corresponds to a polarized variation of Hodge structure, then the mixed Hodge structure on $H^{i-n}(U, DR(\cM))$ is pure of weight $w +i -n$ provided that $i \leq n- n_{S} -2.$ Furthermore,
$$H^{i-n}(U,DR(\cM)) \cong \bigoplus_{p \in \Z} H^{i - n}(U, Gr^{F}_{p}DR(\cM))$$ 
and
$$\dim H^{i -n}(U, Gr^{F}_{-p}DR(\cM)) = \dim H^{ i -n}(U, Gr^{F}_{p +n -w -i}DR(\cM)).$$
\end{thm}

\begin{thm}(Thm. \ref{Hodge})\label{thm.3}
    Suppose $U$ is a smooth irreducible quasi-projective variety of dimension $n$. Let $X$ be a compactification of $U$ and $S = X \backslash U$ with dimension $n_{S}$. Let $\bH= ((H,F^{\bullet}),L)$ be a polarized variation of Hodge structure of weight $w-n$ on $U$, and $\cM \in HM_{U}(U,w)$ the corresponding Hodge module. Let $IC^{H}_{X}(L) \in HM_{X}(X,w)$ be the unique Hodge module on $X$ that extends $\bH,$  whose underlying perverse sheaf is the intersection complex $IC_{X}(L)$ twisted by $L$. The natural map   
     $$H^{i-n}(X,IC_{X}(L)) \rightarrow H^{i}(U,L)$$ 
     is an isomorphism of pure $\mathbb{Q}$-Hodge structures of weight $w +i -n$ whenever $i \leq n- n_{S} -1.$ Furthermore, if $i \leq n -n_{S} -2,$ then the natural map
     $$H^{i-n}(X, Gr^{F}_{-p}DR(IC^{H}_{X}(L))) \rightarrow H^{i-n}(U, Gr^{F}_{-p}DR(\cM))$$
     is an isomorphism. 
 \end{thm}

The second objective of this paper is to show local versions of the previously stated results. Furthermore, most results hold when considering varieties with the associated analytic structure. First, we show there is a local version of Theorem \ref{thm.1}. If we let $\cV =\cL$ in Theorem \ref{thm.1}, where $\cL$ is a very ample sheaf on $X$, then  there is a spectral sequence
$$E^{m,p}_{2}= H^{m}(X, R^{p}f_{\ast}\bR \cH om_{E}(\cF, \omega^{\bullet}_{Y}) \otimes \cL^{N}) \Rightarrow Ext^{p +m}_{E}(\cF \otimes f^{\ast}\cL^{-N}, \omega^{\bullet}_{Y}).$$
By the results of Siu \cite{Siu}, we can prove that the sheaves $R^{p}f_{\ast}\bR \cH om_{E}(\cF, \omega^{\bullet}_{Y})$ are coherent in certain ranges when given specific assumptions on $\cF$. From the result of Theorem \ref{thm.1},  we show the following theorem holds.
\begin{nota}\label{nota1.8}
Given a complex of sheaves $\cF$, we have the truncations
$$\sigma_{>n}\cF: \cdots \rightarrow 0 \rightarrow 0 \rightarrow \cF^{n+1} \rightarrow \cF^{n+2} \rightarrow \cdots$$
$$\sigma_{\leq n}\cF: \cdots \rightarrow \cF^{n-1} \rightarrow \cF^{n} \rightarrow 0 \rightarrow 0 \rightarrow \cdots$$
$$\tau^{>n}\cF: \cdots  \rightarrow 0 \rightarrow 0 \rightarrow \im d^{n} \rightarrow \cF^{n+1} \rightarrow \cdots$$
$$\tau^{\leq n}\cF: \cdots \rightarrow \cF^{n-1} \rightarrow \ker d^{n} \rightarrow 0  \rightarrow 0 \rightarrow \cdots.$$
\end{nota}

\begin{thm}\label{thm.4}(Thm. \ref{prop1.9}) Suppose $X$ is an equidimensional projective variety of dimension $n,$ and $f: Y \rightarrow X$ is a proper birational map.  Let $S \subset X$ be a closed subvariety of dimension $n_{S}$ such that $U = X \backslash S$ is smooth, and $S$ does not contain any irreducible component of $X.$  Let  $E = f^{-1}(S)_{red}$ and  $V = Y \backslash E,$ and assume the induced map $f|_{V}:V \rightarrow U$ is an isomorphism. Let $\cF \in D^{b}_{coh}(\cO_Y)$ and suppose $\cF$ is quasi-isomorphic to a complex $\tcF$ with the following properties: 
\begin{enumerate}
\item  $\sigma_{\leq -n -1}\tcF = 0$ and there exists $l \in [0 ,n]$ such that $ \sigma_{> -n+l}\tcF = 0$.\\

\item There exists $q \geq - n + l +1$ such that $R^{i}f_{\ast}\tcF= 0$ for all $i \geq q$.\\

\item  The sheaves $\tcF^{j}|_{V}$ are locally free on $V$ for all $j$.
\end{enumerate} 
Then  
$$R^{p} f_{\ast}\bR \cH om_{E}(\cF, \omega^{\bullet}_{Y}) = 0 \quad \text{for $p \leq -n_{S} - q$}.$$
\end{thm}

  Suppose $f: Y \rightarrow X$ is a strong log resolution of $S.$ That is, $f: Y \rightarrow X$ is a resolution of singularities such that the induced map $f: f^{-1}(U) \rightarrow U$ is an isomorphism and $E = f^{-1}(S)_{red}$ is a simple normal crossing divisor on $Y.$ Let $\bH= ((\cL, F^{\bullet}),L)$ be a polarized variation of Hodge structure on $V =f^{-1}(U)$ and let $\cM \in HM(V,w)$ be the corresponding Hodge module on $V.$ If $\rho: V \hookrightarrow Y$ is the natural map, then there is a mixed Hodge module $\cM(!E):= \rho_{!}\cM \in MHM(Y)$ such that the complex $Gr^{F}_{p}DR(\cM(!E)) \in D^{b}_{coh}(\cO_{Y})$ is quasi-isomorphic to a complex which satisfies the three conditions in Theorem \ref{thm.4} with $l = n$ and $q = 1.$ Hence we acquire the following proposition.

\begin{prop}\label{propext}(Prop. \ref{prop4.1})
With the same assumptions as above, if $\cM(*E):= \rho_{+}\cM,$ then the natural map
$$R^{i}f_{\ast}Gr^{F}_{-p}DR(\cM(\ast E)) \rightarrow R^{i}j_{\ast}f|_{V \ast}Gr^{F}_{-p }DR(\cM)$$

is an isomorphism for $i \leq -n_{S} -2$ and injective for $i = -n_{S} -1$. 
\end{prop}

In particular, if $\cM$ is the trivial Hodge module $\Q^{H}_{V}[n]$ in the previous proposition, then with the theory of Hodge modules, we have the following results for projective varieties.

\begin{cor}(Cor. \ref{cor4.2})
Let $X$ be an equidimensional projective variety of dimension $n.$  Let $S \subset X$ be a closed subvariety of dimension $n_{S}$ such that $U = X \backslash S$ is smooth, and $S$ does not contain any irreducible component of $X.$ If $f:Y \rightarrow X$ is a strong log resolution of $S$ and $j:U \hookrightarrow X$ is the natural inclusion, then the natural map
$$R^{i}f_{\ast}\Omega^{p}_{Y}(\log E) \rightarrow R^{i}j_{\ast}\Omega^{p}_{U}$$
is an isomorphism for $p +i \leq n - n_{S} -2$ and injective for $p +i = n- n_{S} - 1$. 
\end{cor}

\begin{cor}(Cor. \ref{IC cor 2})
    Let $X$ be an irreducible projective variety and $S \subset X$ be a proper closed subvariety of dimension $n_{S}$ such that $U = X \backslash S$ is smooth. Let $IC^{H}_{X} \in HM_{X}(X,n)$ denote the unique extension of the trivial Hodge module $\Q^{H}_{U}[n] \in HM_{U}(U,n).$ If $j:U \hookrightarrow X$ is the natural inclusion, then the natural map
    $$\cH^{i +p -n}(Gr^{F}_{-p}DR(IC^{H}_{X})) \rightarrow R^{i}j_{*} \Omega^{p}_{U}$$
    is an isomorphism for $p +i \leq n -n_{S} -2$ and injective for $p +i = n -n_{S} -1.$ Furthermore, if $g: Y' \rightarrow X$ is any resolution of singularities, then the natural map
    $$ R^{i}g_{*}\Omega^{p}_{Y'} \rightarrow R^{i}j_{*}\Omega^{p}_{U}$$
    is surjective for $p +i \leq n - n_{S} -2.$
\end{cor}

\subsection*{Acknowledgments} The majority of the results that appear in this paper appeared in the author's PhD thesis \cite{hiattT} at Purdue University, which was directed by Donu Arapura, whose work \cite{arapura} heavily inspired this paper. The author would like to thank Donu Arapura for the helpful conversations, comments, and constant encouragement while writing this paper. Finally, the author would like to thank the referee for careful reading and for many helpful comments.

\section{Local Cohomology}

\begin{set}\label{set1}
 Let $Y$ and $X$ be varieties  of dimension $n_{Y}$ and $n_{X}$, respectively. A variety will always mean a reduced separated scheme of finite type over $\C$. We will consider maps $f:Y \rightarrow X$ and assume all maps $f:Y \rightarrow X$ to be surjective and proper. Let $S$ be a closed subset of $X$ of dimension $n_{S}$ and $E = f^{-1}(S)_{red} \subset Y$. We will denote their natural maps as $i_{S}:S \hookrightarrow X$ and $\iota:E \hookrightarrow Y$. We will also denote $U = X \backslash S$ and $V =f^{-1}(U)$ with natural maps $j:U \hookrightarrow X$ and $\rho:V \hookrightarrow Y$.

With the given notation, we have the commutative diagram.

$$
\adjustbox{scale=1.25,center}{\begin{tikzcd}
E \arrow[r, "\iota"] \arrow[d, "f|_{E}"] & Y \arrow[d, "f"] & V \arrow[l, "\rho"] \arrow[d, "f|_{V}"] \\
S \arrow[r, "i_{S}"] & X & U \arrow[l, "j"].
\end{tikzcd}}$$

In the case $f: Y \rightarrow X$ is a resolution of singularities, we will always assume $S$ contains the singular locus of $X,$ and $S$ does not include any irreducible component of $X.$ By a strong log resolution of $S$, we mean $f: Y \rightarrow X$ is a resolution of singularities such that the induced map $f|_{V}:V \rightarrow U$ is an isomorphism, and $E$ is a simple normal crossing divisor on $Y.$ Since we work with complex varieties, such resolutions exist, \cite{hironaka}.
\end{set}

\begin{nota}\label{nota1} The right derived functor of an additive functor $F$, if it exists, will be denoted as $\bR F$. For a closed subset $\mathcal{Z}$, we will denote $\underline{\Gamma_{\mathcal{Z}}}$ as the sheaf of sections with support in $\mathcal{Z}$ and denote the global sections with support in $\mathcal{Z}$ as  $\Gamma_{\mathcal{Z}}$. The $i^{th}$ cohomology of the right derived functors $\bR \underline{\Gamma_{\mathcal{Z}}}$ and $\bR \Gamma_{\mathcal{Z}}$ will be denoted as $\cH^{i}_{\mathcal{Z}}$ and $H^{i}_{\mathcal{Z}}$, respectively. The composition $\bR \Gamma_{\mathcal{Z}} \circ \bR Hom(\cF, \bullet)$ will be denoted as $\bR Hom_{\mathcal{Z}}(\cF, \bullet)$ and the $i^{th}$ cohomology as $Ext^{i}_{\mathcal{Z}}(\cF, \bullet)$. 
\end{nota}


\begin{lemma}\label{lemma1.3} Given $\cF \in D_{qcoh}(\cO_Y)$, the natural maps
$$\textbf{R}f_{\ast}\textbf{R}\rho_{\ast}\rho^{\ast}\mathcal{F} \rightarrow \textbf{R}j_{\ast}j^{\ast}\textbf{R}f_{\ast}\mathcal{F} \quad \text{and} \quad
\textbf{R}f_{\ast}\textbf{R}\underline{\Gamma_{E}}\mathcal{F} \rightarrow \textbf{R}\underline{\Gamma_{S}}\textbf{R}f_{\ast}\mathcal{F}$$
are quasi-isomorphisms.
\end{lemma}

\begin{proof}  Since the maps $j$ and $\rho$ are both flat, there is a natural functorial isomorphism $j^{\ast}\textbf{R}f_{\ast}\mathcal{F} \simeq \textbf{R}f|_{V \ast}\rho^{\ast}\mathcal{F}$ \cite[Prop II.5.12]{Hart1}. Therefore, we have quasi-isomorphisms
\begin{center}
 $\textbf{R}f_{\ast}\textbf{R}\rho_{\ast}\rho^{\ast}\mathcal{F} \simeq \textbf{R}j_{\ast}\textbf{R}f|_{V \ast}\rho^{\ast}\mathcal{F} \simeq  \textbf{R}j_{\ast}j^{\ast}\textbf{R}f_{\ast}\mathcal{F}$.
\end{center} 
Which proves the first quasi-isomorphism.

There is a map of triangles with the natural maps,
\begin{center}
\begin{tikzcd}
\textbf{R}f_{\ast}\textbf{R}\underline{\Gamma_{E}}\mathcal{F} \arrow[d] \arrow[r] & \textbf{R}f_{\ast}\mathcal{F} \arrow[r] \arrow[d, "\parallel "] & \textbf{R}f_{\ast}\textbf{R}\rho_{\ast}\rho^{\ast}\mathcal{F} \arrow[d, "\wr "] \arrow[r, "+1"] & \hfill\\
\textbf{R}\underline{\Gamma_{S}}\textbf{R}f_{\ast}\mathcal{F} \arrow[r] & \textbf{R}f_{\ast}\mathcal{F} \arrow[r] & \textbf{R}j_{\ast}j^{\ast}\textbf{R}f_{\ast}\mathcal{F} \arrow[r, "+1"] & \hfill.
\end{tikzcd}
\end{center}
The map $ \textbf{R}f_{\ast}\textbf{R}\underline{\Gamma_{E}}\mathcal{F} \rightarrow \textbf{R}\underline{\Gamma_{S}}\textbf{R}f_{\ast}\mathcal{F} $ must necessarily be a quasi-isomorphism from this map of triangles.  
\end{proof}

\begin{prop}\label{prop1.4} Let  $\omega^{\bullet}_{X}$ and $\omega^{\bullet}_{Y}$ denote dualizing complexes of $X$ and $Y$, respectively. If $\mathcal{F} \in D^{b}_{ qcoh}(\cO_Y)$, then there is a quasi-isomorphism
\begin{center}
$\textbf{R}f_{\ast}\textbf{R}\underline{\Gamma_{E}}\bR \mathcal{H}om(\mathcal{F}, \omega_{Y}^{\bullet}) \simeq \textbf{R}\underline{\Gamma_{S}}\bR\mathcal{H}om(\textbf{R}f_{\ast}\mathcal{F}, \omega_{X}^{\bullet})$.
\end{center}
\end{prop}

\begin{proof} By Grothendieck duality, we have the quasi-isomorphism $\textbf{R}f_{\ast}\textbf{R}\mathcal{H}om(\mathcal{F}, \omega^{\bullet}_{Y})  
\simeq \textbf{R}\mathcal{H}om(\textbf{R}f_{\ast}\mathcal{F}, \omega^{\bullet}_{X})$. If we apply $\textbf{R}\underline{\Gamma_{S}}$ to both sides and use Lemma \ref{lemma1.3}, then we have the desired quasi-isomorphism. 
\end{proof}

\begin{thm}\label{thm1.5}
Assume $X$ is proper and $\mathcal{F} \in D^{b}_{coh}(\cO_{Y})$. If there exists $q \in \Z$ such that $R^{i}f_{\ast}\mathcal{F} = 0$ for all $i \geq q$, then
\begin{center}
$Ext^{p}_{E}(\mathcal{F} \otimes f^{*}\cV, \omega^{\bullet}_{Y}) = 0$ for $ p \leq -n_{S} -q $,
\end{center}
where $\cV$ is any locally free sheaf on $X.$
\end{thm}

\begin{proof}
  If we apply $\textbf{R}\Gamma(X, \bullet)$ to Proposition \ref{prop1.4}, then there is a quasi-isomorphism
 \begin{equation}
\textbf{R}Hom_{E}(\mathcal{F} \otimes f^{*}\cV, \omega^{\bullet}_{Y}) \simeq \textbf{R}Hom_{S}(\textbf{R}f_{\ast}\mathcal{F} \otimes \cV, \omega^{\bullet}_{X}).
\end{equation}
If $\mathcal{X}$ is the completion of $X$ along $S$ and $\kappa: \mathcal{X} \rightarrow X$ is the natural map, then by \cite[5.3]{AJL} we have the quasi-isomorphism
$$\textbf{R}\Gamma(\mathcal{X}, \kappa^{\ast}(\textbf{R}f_{\ast}\mathcal{F} \otimes \cV)) \simeq Hom_{\C}(\textbf{R}Hom_{S}(\textbf{R}f_{\ast}\mathcal{F} \otimes \cV, \omega^{\bullet}_{X}), \C).$$
By equation (3), we receive the quasi-isomorphism
\begin{equation}\textbf{R}\Gamma(\mathcal{X}, \kappa^{\ast}(\textbf{R}f_{\ast}\mathcal{F} \otimes \cV)) \simeq  Hom_{\C}(\textbf{R}Hom_{E}(\mathcal{F} \otimes f^{*}\cV, \omega^{\bullet}_{Y}),\C) .
\end{equation}
Let $S_{k}$ be the $k$th infinitesimal neighborhood of $S$. By \cite[Prop. 0.13.3.1]{EGA}, we have the isomorphism
$$H^{a}(\mathcal{X}, \kappa^{\ast}(R^{b}f_{\ast}\mathcal{F} \otimes \cV)) \cong \displaystyle \varprojlim_{k} H^{a}(S_{k}, (R^{b}f_{\ast}\mathcal{F} \otimes \cV) \otimes \mathcal{O}_{S_{k}}).$$
Recall that the dimension of $S$ is $n_{S}$ and we are assuming $R^{i}f_{\ast}\mathcal{F} = 0$ for all $i \geq q$. By the isomorphism above,
 $$H^{a}(\mathcal{X}, \kappa^{\ast}(R^{b}f_{\ast}\mathcal{F} \otimes \cV)) = 0 \quad  \text{whenever $ a + b \geq n_{S} + q$}.$$
Using the Leray spectral sequence, we obtain
 $$H^{p}(\mathcal{X}, \kappa^{\ast}(\textbf{R}f_{\ast}\mathcal{F} \otimes \cV)) = 0 \quad \text{for $p \geq n_{S} + q$}.$$
  Now, using equation (4), we have
$$Hom_{\C}(Ext^{-p}_{E}(\mathcal{F} \otimes f^{*}\cV, \omega^{\bullet}_{Y}),\C) \cong H^{p}(\mathcal{X}, \kappa^{\ast}(\textbf{R}f_{\ast}\mathcal{F} \otimes \cV))  = 0 \quad \text{for $p \geq n_{S} +q$}.$$
Hence $Ext^{p}_{E}(\mathcal{F} \otimes f^{*}\cV, \omega^{\bullet}_{Y}) = 0$ for $  p \leq -n_{S} - q$. 

\end{proof}

\begin{cor}\label{cor1.6} Assume $X$ is proper, and $Y$ is smooth and irreducible. If $\mathcal{F}$ is a locally free sheaf on $Y$ and there exists $q \in \Z$ such that $R^{i}f_{\ast}\mathcal{F} = 0$ for all $i \geq q$, then 
$$H^{p}_{E}(Y, \mathcal{F}^{\vee} \otimes f^{\ast}\cV\otimes \Omega^{n_{Y}}_{Y}) = 0 \quad \text{for $p \leq n_{Y} - n_{S} -q  $},$$
where $\mathcal{F}^{\vee} = \mathcal{H}om(\mathcal{F}, \mathcal{O}_{Y})$ and $\cV$ is any locally free sheaf on $X$. 
\end{cor}

\begin{proof}
If $Y$ is smooth and irreducible, then $\omega^{\bullet}_{Y} \simeq \Omega^{n_{Y}}[n_{Y}]$. Therefore,
$$Ext^{-p}_{E}(\mathcal{F} \otimes f^{*}\cV^{\vee}, \omega^{\bullet}_{Y}) \cong Ext^{n_{Y}-p}_{E}(\cO_{Y}, \mathcal{F}^{\vee} \otimes f^{*}\cV \otimes \Omega^{n_{Y}}_{Y}) \cong H_{E}^{n_{Y}-p}(Y, \mathcal{F}^{\vee} \otimes f^{*}\cV \otimes \Omega^{n_{Y}}_{Y}).$$
Now apply Theorem \ref{thm1.5}.
\end{proof}

 If $\cF = \Omega^{n-q}_{Y}(\log E)(-E)$, then $R^{i}f_{\ast}\Omega^{n-q}_{Y}(\log E)(-E)= 0$ for all $i \geq q +1$ \cite{steenbrink}. By the previous result, we obtain Arapura's vanishing theorem. 

\begin{cor}\label{cor1.7}\cite{arapura}
If $X$ is an irreducible projective variety and $f:Y \rightarrow X$ is a strong log resolution of $S,$ then
$$H^{p}_{E}(Y, \Omega^{q}_{Y}(\log E) \otimes f^{\ast}\cV) = 0 \quad \text{for $p + q \leq n_{Y} - n_{S} -1$},$$
where $\cV$ is any locally free sheaf on $X$. 
\end{cor}

\section{Hodge Modules}
We will briefly review some of the main properties and theorems of Hodge modules. If $X$ is a smooth algebraic variety of dimension $n$, we have the following categories:
\begin{enumerate}

\item $HM(X,w) := $ category of polarizable Hodge modules of weight $w$ on $X$.

\item $HM_{Z}(X,w) :=$ category of polarizable Hodge modules of weight $w$ on $X$ with strict support on a closed irreducible subvariety $Z \subseteq X$.

\item $MHM(X):=$ category of polarizable mixed Hodge modules. 

\end{enumerate} 
Roughly speaking, a mixed Hodge module $\cN$ consists of the following data:
\begin{itemize}

\item A right regular holonomic $\cD_{X}$-module $N$ with a good filtration  $F_{\bullet}N$. The filtration $F_{\bullet}N$ is called the Hodge filtration. 

\item A weight filtration $W_{\bullet}\cN$, where $Gr^{W}_{i}\cN$ is a pure Hodge module of weight $i$ on $X$.

\item A perverse sheaf $K$ with $\Q$-coefficients on $X$  such that 
$$ K \otimes \C \simeq DR(\cN): = [N \otimes \wedge^{n} \Theta_{X} \rightarrow N \otimes \wedge^{n-1}\Theta_{X} \rightarrow \cdots \rightarrow N][n],$$ 
where $\Theta_{X}$ is the sheaf of vector fields on $X$. This complex is called the de Rham complex of the $\cD_{X}$-module $N$. 
\end{itemize} 
For a mixed Hodge module $\cN \in MHM(X)$, we will denote the above data as 
$$\cN = ((N, F_{\bullet}N), W_{\bullet}\cN, K).$$ 
The weight filtration will be omitted if $\cN$ is pure of weight $w$.

\begin{rmk}
    All variations of Hodge structures and Hodge modules are assumed to be polarizable.
\end{rmk}

\begin{rmk}The definition of mixed Hodge modules (in the analytic category) also requires information on the nearby and vanishing cycles along locally defined holomorphic functions. As well as an $admissible$ condition along locally defined holomorphic functions. For more information on this, see \cite{saito},\cite{saito2}, or \cite{schnell}. 
\end{rmk}

The first condition that is imposed on $HM(X,w)$ is the condition of strict support. For any $\cM \in HM(X,w),$ there is a decomposition
$$ \cM = \bigoplus_{Z} \cM_{Z} \quad \text{for $\cM_{Z} \in HM_{Z}(X,w)$}.$$
Therefore, to understand a general Hodge module $\cM \in HM(X,w)$, it suffices to understand its direct factors $\cM_{Z} \in HM_{Z}(X,w)$. The next theorem, which can be viewed as the $fundamental$ $theorem$ $of$ $Hodge$ $modules$, explicitly describes the category $HM_{Z}(X,w)$.

\begin{thm}\cite[Theorem 1.3]{saito4} For any closed irreducible subvariety $Z \subset X$, the restriction to sufficiently small open subvarieties of $Z$ induces an equivalence of categories
$$MH_{Z}(X,w) \simeq VHS_{gen}(Z, w - dim(Z))^{p},$$
the right-hand side is the category of polarizable variations of pure Hodge structure of weight $w - dim(Z)$ defined on a smooth, dense open subvariety $U$ of $Z$. Moreover, this equivalence of categories induces a one-to-one correspondence between polarizations of $\cM \in HM_{Z}(X,w)$ and those of the corresponding generic variation of Hodge structure.
\end{thm}

\begin{defn}\label{VHS} A variation of Hodge structure $\bH = ((H,F^{\bullet}),L)$ of weight $k$ on $X$ consists of the following data:
\begin{enumerate}
\item a local system $L$ over $\mathbb{Q}$ on $X;$
\item a finite decreasing filtration $\{F^{p}\}$ of the holomorphic vector bundle  $H:= L \otimes_{\mathbb{Q}}\cO_{X}$ by holomorphic subbundles.
\end{enumerate}
These data should satisfy the following conditions:
\begin{enumerate}
    \item for each $x \in X$ the filtration $\{F^{p}(x)\}$ of $H(x) \cong (L \otimes_{\mathbb{Q}}\cO_{X}) \otimes_{\cO_{X}}k(x)$ defines a Hodge structure of weight $k,$ where $k(x)$ is the residue field of $x$;
    \item The connection $\nabla:H\rightarrow H \otimes_{\cO_{X}}\Omega^{1}_{X}$ whose sheaf of horizontal sections in $L$ satisfies the Griffiths' transversality condition
    $$\nabla(F^{p}) \subset F^{p-1}\otimes \Omega^{1}_{X}.$$
    \end{enumerate}
    \end{defn}

    \begin{rmk}
        The holomorphic vector bundle $H$ from a variation of Hodge structure is a left regular holonomic $\cD_{X}$-module with a good filtration defined by $F_{p}H:= F^{-p}H.$ Since we are working with right $\cD_{X}$-modules, we have the right $\cD_{X}$-module $H \otimes_{\cO_{X}}\Omega^{n}_{X}$ with filtration $F_{p-n}(H \otimes \Omega^{n}_{X}):= F_{p}H \otimes \Omega^{n}_{X}.$
    \end{rmk}

\begin{ex}
Assume $X$ is irreducible, and let $\Q^{H}_{X}[n]:=((\Omega^{n}_{X}, F_{\bullet}\Omega^{n}_{X}), \Q_{X}[n])$, where 
$$ F_{i}\Omega^{n}_{X} = 
\begin{cases}
\Omega^{n}_{X} & \text{if $i \geq -n$} \\
0 & \text{otherwise}.
\end{cases}$$
By Theorem 1.3, $\Q^{H}_{X}[n]$ is a pure Hodge module of weight $n$ because it is the constant variation of Hodge structure on $X$.
\end{ex}

\subsection{The Hodge Filtration}
 From the Hodge filtration $F_{\bullet}N,$ we have a natural filtration on the de Rham complex given by:
\begin{equation} 
F_{p}DR(\cN) :=[F_{p-n}N \otimes \wedge^{n}\Theta_{X} \rightarrow F_{p-n +1}N \otimes \wedge^{n-1}\Theta_{X} \rightarrow \cdots \rightarrow F_{p}N][n].
\end{equation} 
Hence there is an associated graded complex on $DR(\cN)$ give by: 
\begin{equation}\label{gradedEQ}
Gr^{F}_{p}DR(\cN) := [Gr^{F}_{p-n}N \otimes \wedge^{n}\Theta_{X} \rightarrow Gr^{F}_{p-n+1}N \otimes \wedge^{n-1}\Theta_{X} \rightarrow \cdots \rightarrow Gr^{F}_{p}N][n]. 
\end{equation}
Since the Hodge filtration is a good filtration on the underlying $\cD_{X}$-module of $\cN$, we have 
$$Gr^{F}_{p}DR(\cN) \in D^{b}_{coh}(\cO_{X}).$$

\begin{ex}
Consider the trivial Hodge module $\Q^{H}_{X}[n]$. If $0 \leq p \leq n$, then it easy to verify the equality
$$Gr^{F}_{-p}DR(\Q^{H}_{X}[n]) = \Omega^{p}_{X}[n-p].$$
\end{ex}

\subsection{Duality}
An important aspect of the theory of mixed Hodge modules, which will be used consistently in this paper, is the feature of duality. In the category of mixed Hodge modules, there is a duality functor 
$$\D: D^{b}MHM(X) \rightarrow D^{b}MHM(X)^{op},$$
where $D^{b}MHM(X)$ is the bounded derived catogory of mixed Hodge modules on $X$. If $\cN = ((N, F_{\bullet}N), W_{\bullet}\cN, K) \in MHM(X)$, then $\D(\cN) \in MHM(X)$ is called the $dual$ $Hodge$ $module$, where the underlying $\cD_{X}$-module of $\D(\cN)$ is given by the dual $\cD_{X}$-module. By the Reimann-Hilbert correspondence, the underlying perverse sheaf of $\D(\cN)$ is given by Verdier dual of $K$. 
 
By equation \ref{gradedEQ}, there are complexes $Gr^{F}_{p}DR(\cN)$ and $Gr^{F}_{p}DR(\D(\cN))$ given from the mixed Hodge modules $\cN$ and $\D(\cN)$,  respectively. The following proposition explains the relationship between these two complexes. 
 
\begin{prop}\label{dualprop}\cite{saito}
If $\cN \in MHM(X)$, then we have the quasi-isomorphism
$$ \bR \cH om_{X}(Gr^{F}_{p}DR(\cN), \omega^{\bullet}_{X}) \simeq Gr^{F}_{-p}DR(\D(\cN)) \quad  \text{for every $p \in \Z$}, $$ 
where $\D(\cN)$ is the dual Hodge module.
\end{prop}

If $\cN \in HM(X,w)$, then a polarization gives an isomorphism $\D(\cN) \cong \cN(w)$, where 
$$\cN(w)= ((N, F_{\bullet -w}N), W_{\bullet-2w}\cN, K \otimes_{\Q} \Q(w))$$
is the Tate twist of $\cN$ by the weight $w$. Therefore, if $\cN \in HM(X,w)$, we have the following immediate corollary of Proposition \ref{dualprop}.

\begin{cor}
If $\cN \in HM(X,w)$, then there is a quasi-isomorphism
$$\bR \cH om_{X}(Gr^{F}_{p}DR(\cN), \omega^{\bullet}_{X}) \simeq Gr^{F}_{-p -w}DR(\cN) \quad \text{for every $p \in \Z$}.$$
\end{cor}

\subsection{Direct Images}
Consider a map $f: X \rightarrow X'$ between smooth algebraic varieties $X$ and $X'$. The direct image functor,
$$f_{+}: D^{b}MHM(X) \rightarrow D^{b}MHM(X'),$$
 is compatible with the direct image functor for $\cD_{X}$-modules and the direct image functor for constructible sheaves.  Furthermore, when the map $f:X \rightarrow X'$ is proper, the functor $Gr^{F}_{p}DR( \bullet)$ commutes with the direct image functor in the following way.

\begin{prop}\cite[Prop. 4.10]{ks}
Let $f: X \rightarrow X'$ be a proper map between smooth algebraic varieties. For every $p \in \Z$, one has a natural isomorphism of functors
$$\bR f_{\ast} \circ Gr^{F}_{p}DR(\bullet) \simeq Gr^{F}_{p}DR \circ f_{+}(\bullet),$$

as functors from $D^{b}(MHM(X))$ to  $D^{b}_{coh}(\cO_{X'})$. 
\end{prop}

\subsection{Hodge Modules on Singular Varieties}
 Whenever $X$ is singular, we assume $X$ embedds into a smooth variety $\tX$. If $$\bT i:X \arrow[r, hook] & \tX \eT$$
 is the natural inclusion map, then we have an equivalence of categories:
$$ \bT i_{+}: D^{b}MHM(X) \arrow[r] & D^{b}MHM_{X}(\tX), 
\eT$$
where $D^{b}MHM_{X}(\tX)$ is the full subcategory of $D^{b}MHM(\tX)$ whose objects have cohomological supports in $X$ \cite{saito2}. For any $\cN \in D^{b}MHM(X)$, the complex $Gr^{F}_{p}DR(i_{X+}\cN) \in D^{b}_{coh}(\cO_{\tX})$ is actually a well defined complex in $D^{b}_{coh}(\cO_X),$ and is independent of the chosen embedding $i:X \hookrightarrow \tX$. Therefore, we have the quasi-isomorphism
$$\bR \Gamma(\tX,  Gr^{F}_{p}DR(i_{X+}\cN)) \simeq \bR \Gamma(X,  Gr^{F}_{p}DR(i_{X+}\cN)).$$

\section{ Kodaira-Saito Type Vanishing for Quasi-Projective Varieties}

We will keep the same setting as \ref{set1}. Assume $f:Y \rightarrow X$ is a resolution of singularities and $f|_{V}:V \rightarrow U$ is an isomorphism. Recall that $U$ is assumed to be smooth in this situation.  Since $n_{X} = n_{Y}$, we will denote $n= n_{X}=n_{Y}$. Further, assume that there exists a smooth ambient variety $\tX$ of $X$ with an embedding $i_{X}: X \hookrightarrow \tX$. We will denote $\tf = i_{X} \circ f.$

\subsection*{A Generalization of Steenbrink's Vanishing Theorem}

Given $\cM \in MHM(V)$ there are two distinguished extensions
$$\cM(!E) := \rho_{!}\cM \quad \text{and} \quad \cM(\ast E) := \rho_{+}\cM.$$
Using the assumption that $f|_{V}:V \rightarrow U$ is an isomorphism, we first prove a lemma concerning the cohomology of $f_{+}\cM(* E)$ and $f_{+}\cM(!E)$, which was inspired  \cite[Lemma 9.3]{ks}.

\begin{lemma}\label{lemma2.1}
If $\cM \in MHM(V)$, then 
$$\tau^{\leq -1}f_{+}\cM(\ast E) = 0 \quad \text{and} \quad  \tau^{\geq 1}f_{+}\cM(!E) = 0. $$
\end{lemma}

\begin{proof}
 By duality, it suffices to show $\tau^{\leq -1}f_{+}\cM(\ast E) = 0$. Furthermore, it suffices to show the underlying constructible complex $^{p}\tau^{\leq -1}\text{rat}(f_{+}\cM(\ast E)) = 0$. The constructible complex for $f_{+}\cM(\ast E)$ is given by
$$\bR f_{\ast}\bR \rho_{\ast}L \cong \bR j_{\ast}f|_{V \ast}L,$$
where $L$ is the underlying perverse sheaf for $\cM$. Abusing notation, we will let $f|_{V \ast}L = L$. 

If $K \in D^{b}_{c}(\Q_{X})$ is any constructible complex, then
$$ Hom_{D^{b}_{c}(\Q_{X})}(K, \bR j_{\ast}L) \cong Hom_{D^{b}_{c}(\Q_{U})}(j^{-1}K, L).$$
If we take $K$ to be $^{p}\tau^{\leq -1} \bR j_{\ast}L$, then we have the isomorphisms 
$$ Hom_{^{p}D^{\leq -1}}(^{p}\tau^{\leq -1} \bR j_{\ast}L, \hspace{.01in}^{p}\tau^{\leq -1} \bR j_{\ast}L) \cong  Hom_{D^{b}_{c}(\Q_{X})}(^{p}\tau^{\leq -1} \bR j_{\ast}L, \bR j_{\ast}L)$$
$$ \cong Hom_{D^{b}_{c}(\Q_{U})}(j^{-1}(^{p}\tau^{\leq -1} \bR j_{\ast}L), L).$$
However, the complex $j^{-1}(^{p}\tau^{\leq -1} \bR j_{\ast}L) =  (^{p}\tau^{\geq -1}L) = 0$. Therefore,
$$Hom_{^{p}D^{\leq -1}}(^{p}\tau^{\leq -1} \bR j_{\ast}L, ^{p}\tau^{\leq -1} \bR j_{\ast}L) \cong Hom_{D^{b}_{c}(\Q_{U})}(j^{-1}(^{p}\tau^{\leq -1} \bR j_{\ast}L), L) = 0.$$
This forces $^{p}\tau^{\leq -1} \bR j_{\ast}L = 0$. 
\end{proof}

\begin{thm}\label{thm2.2} Assume $\cM \in MHM(V)$ and $X$ is an irreducible projective variety. If $\cL$ is an ample invertible sheaf on $X$, then
$$H^{i}(X, Gr^{F}_{p}DR(\tf_{+}\cM(\ast E)) \otimes \cL^{-1}) = 0 \quad \text{ for all $i< 0$ and $p \in \Z$} $$
and
$$H^{i}(X, Gr^{F}_{p}DR(\tf_{+}\cM(! E)) \otimes \cL) = 0 \quad \text{ for all $i> 0$ and $p \in \Z$}. $$

\end{thm}

\begin{proof}
 By duality, it suffices to prove the first statement. From the previous lemma, there is an exact triangle
$$\begin{tikzcd}
\cH^{0}(\tf_{+}\cM(\ast E)) \arrow[r] & \tf_{+}\cM(\ast E) \arrow[r] & \tau^{\geq 1} \tf_{+}\cM(\ast E)) \arrow[r, "+1"] & \hfill. 
\end{tikzcd}$$
 By Kodaira-Saito vanishing \cite[Prop. 2.33]{saito2},  $H^{i}(X, Gr^{F}_{p}DR(\cH^{0}(\tf_{+}\cM(\ast E)) \otimes \cL^{-1})= 0$ for $i < 0$. By the distinguished triangle above, it suffices to show 
$$H^{i}(X,  Gr^{F}_{p}DR(\tau^{\geq 1} \tf_{+}\cM(\ast E))\otimes \cL^{-1}) = 0 \quad \text{for $i < 0$}. $$
Now consider the distinguished triangle
$$\begin{tikzcd}
\cH^{1}(\tf_{+}\cM(\ast E))[-1] \arrow[r] & \tau^{\geq 1}\tf_{+}\cM(\ast E) \arrow[r] & \tau^{\geq 2}\tf_{+}\cM(\ast E) \arrow[r, "+1"] & \hfill.
\end{tikzcd}$$
By Kodaira-Saito vanishing, $H^{i}(X, Gr^{F}_{p}DR(\cH^{1}(\tf_{+}\cM(\ast E))[-1] \otimes \cL^{-1})= 0$ for $i < 0$.
So, it suffices to show
$$H^{i}(X,  Gr^{F}_{p}DR(\tau^{\geq 2} \tf_{+}\cM(\ast E))) \otimes \cL^{-1}) = 0 \quad \text{for $i < 0$}.$$
Continue in this pattern. Since $\tf_{+}\cM(\ast E)$ is bounded, there exists $i_{0} \in \N$ such that $\tau^{\geq i_{0}}\tf_{+}\cM(\ast E) = 0.$
For dimension reasons we actually have $\tau^{\geq n} \tf_{+}\cM(\ast E) = 0$ \cite[Prop. 8.1.42]{htt}. Therefore, 
$$\tau^{\geq n-1} \tf_{+}\cM(\ast E) \cong \cH^{n-1}(\tf_{+}\cM(\ast E))[-n +1]$$ 
and
$$H^{i}(X,  Gr^{F}_{p}DR(\tau^{\geq n-1} \tf_{+}\cM(\ast E)) \otimes \cL^{-1}) = 0 \quad \text{for $i< 0$}.$$
The theorem is held by exhaustion.

\end{proof}

\begin{cor}\label{cor2.3}
 If $\cM \in MHM(V)$ and $X$ is a projective variety, then
$$ R^{i}f_{\ast}Gr^{F}_{p}DR(\cM(! E)) = 0 \quad \text{for $i >0$ and $p \in \Z$}.$$ 
\end{cor}

\begin{proof}
Without loss of generality, we may assume $X$ is irreducible. Since $\tf: Y \rightarrow \tX$ is proper, there is a quasi-isomorphism
$$i_{X*}\bR f_{\ast}Gr^{F}_{p}DR(\cM(!E)) \simeq Gr^{F}_{p}DR(\tf_{+}\cM(! E)).$$
From this quasi-isomorphism, there is a spectral sequence
$$H^{j}(X, R^{i}f_{\ast}Gr^{F}_{p}(\cM(!E)) \otimes \cL^{N} ) \Rightarrow H^{i +j}(X, Gr^{F}_{p}DR(\tf_{+}\cM(! E)) \otimes \cL^{N}),$$
where $\cL$ is an ample invertible sheaf on $X$ and $N \in \Z$. When $N \gg 0$ and $j > 0$ we obtain 
$$H^{j}(X, R^{i}f_{\ast}Gr^{F}_{p}(\cM(! E)) \otimes \cL^{N} ) = 0.$$ Therefore,
$$\Gamma(X,  R^{i}f_{\ast}Gr^{F}_{p}(\cM(! E)) \otimes \cL^{N} ) \cong   H^{i}(X, Gr^{F}_{p}DR(\tf_{+}\cM(! E)) \otimes \cL^{N}).$$
By Theorem \ref{thm2.2}, 
$$\Gamma(X,  R^{i}f_{\ast}Gr^{F}_{p}(\cM(! E)) \otimes \cL^{N} ) \cong H^{i}(X, Gr^{F}_{p}DR(\tf_{+}\cM(! E)) \otimes \cL^{N}) = 0 \quad  \text{for $i >0$}.$$
 For $N \gg 0$ the sheaf $R^{i}f_{\ast}Gr^{F}_{p}(\cM(! E)) \otimes \cL^{N} $ is globally generated. From the isomorphism above we conclude $R^{i}f_{\ast}Gr^{F}_{p}DR(\cM(! E)) = 0$ for $i > 0$. 
\end{proof}

If we choose $\cM$ to be the trivial Hodge module $\Q^{H}_{V}[n],$ then we receive Steenbrink's vanishing theorem.

\begin{cor}\label{cor2.4}\cite{steenbrink}
Assume $X$ is an irreducible projective variety and $\cL$ is an ample invertible sheaf on $X$. If $f:Y \rightarrow X$ is a strong log resolution of $S,$ then 

$$H^{q}(Y, \Omega^{p}_{Y} (\log E)\otimes f^{\ast}\cL^{-1}) = 0 \quad \text{for $p +q< n,$}$$

$$H^{q}(Y, \Omega^{p}_{Y}(\log E)(-E) \otimes f^{\ast}\cL) = 0 \quad \text{for $p +q > n,$}$$
and
$$R^{q}f_{\ast}\Omega_{Y}^{p}(\log E)(-E) = 0 \quad \text{for $p +q >n$}.$$
\end{cor}

\begin{proof}
 If $\cM = \Q^{H}_{V}[n]$ is the trivial Hodge module on $V$, then there is a filtered quasi-isomorphism 
$$DR(\rho_{+}\Q^{H}_{V}[n]) \simeq (\Omega^{\bullet}_{Y}(\log E),F_{\bullet})[n],$$
where the filtration on $\Omega^{\bullet}_{Y}(\log E)$ is defined by $F_{p} = \sigma_{\geq -p}$ \cite[Prop 3.11]{saito2}. For $0 \leq p \leq n$ there is a quasi-isomorphism
$$Gr^{F}_{-p} DR(\rho_{+}\Q^{H}_{V}[n]) \simeq \Omega^{p}_{Y}(\log E)[n - p].$$
The map $\tf: Y \rightarrow \tX$ gives us
$$i_{X*}\bR f_{\ast}\Omega^{p}_{Y}(\log E)[n- p] \simeq Gr^{F}_{-p} DR(\tf_{+}\rho_{+}\Q^{H}_{V}[n]).$$
 If $\cL$ is an  ample invertible sheaf, we have the quasi-isomorphisms 
 $$\bR \Gamma(X,Gr^{F}_{-p} DR(\tf_{+}\rho_{+}\Q^{H}_{V}[n]) \otimes \cL^{-1}) \simeq  \bR \Gamma(X,\bR f_{\ast}\Omega^{p}_{Y}(\log E)[n - p] \otimes \cL^{-1}) $$
 $$\simeq \bR \Gamma(Y, \Omega^{p}_{Y}(\log E)[n - p] \otimes f^{\ast}\cL^{-1}).$$ 
 Applying Theorem \ref{thm2.2} we obtain
$$H^{i}(Y, \Omega^{p}_{Y} (\log E)[n -p]\otimes f^{\ast}\cL^{-1}) = 0 \quad \text{for $i < 0$}$$
and
$$H^{i}(Y, \Omega^{n-p}_{Y}(\log E)(-E)[p] \otimes f^{\ast}\cL) = 0 \quad \text{for $i > 0$}.$$
By Corollary \ref{cor2.3},
$$R^{i}f_{\ast}\Omega_{Y}^{n -p}(\log E)(-E)[p] = 0 \quad \text{for $i >0$}.$$
Adjusting for the shifts, we have the desired result.
\end{proof}

\subsection*{Kodaira-Saito Type Vanishing for Quasi-Projective Varieties}

Using Corollary \ref{cor2.3}, we can apply Theorem \ref{thm1.5} to prove the vanishing of local cohomology for specific mixed Hodge modules. We still assume $f:Y \rightarrow X$ is a resolution of singularities, $U$ is smooth, and $f|_{V}:V \rightarrow U$ is an isomorphism. For the rest of this section, assume $X$ is an irreducible projective variety. Using the isomorphism $f|_{V}:V \rightarrow U$, we will abuse notation to let $\cM$ denote a mixed Hodge module on $V$ or $U$.

\begin{prop}\label{prop2.5}
If $\cV$ is locally free on $X$, then
$$ H^{p}_{E}(Y, Gr^{F}_{q}DR(\cM(\ast E)) \otimes f^{\ast}\cV) = 0 \quad \text{for $p \leq -n_{S} - 1$}.$$

\end{prop}

\begin{proof}
Let $\D(\cM)$ denote the dual of $\cM$. By Corollary \ref{cor2.3}, $ R^{i}f_{\ast}Gr^{F}_{-q}DR(\D(\cM)(! E)) = 0$ for $i \geq 1$. So we may apply Theorem \ref{thm1.5} with $q =1$. For $p \leq -n_{S} -1$, we obtain
$$H^{p}_{E}(Y, Gr^{F}_{q}DR(\cM(\ast E)) \otimes f^{\ast}\cV) \cong Ext^{p}_{E}(Gr^{F}_{-q}DR(\D(\cM)(!E)) \otimes f^{\ast}\cV^{\vee}, \omega^{\bullet}_{Y}) = 0. $$
\end{proof}

\begin{rmk}
If  $f:Y \rightarrow X$ is a strong log resolution and we let $\cM(\ast E) = \rho_{+}\Q^{H}_{V}[n]$ in the previous proposition, then we obtain Arapura's vanishing theorem \cite[Thm. 1]{arapura}
\end{rmk}

\begin{cor}\label{cor2.7}
If  $\cV$ is locally free on $X$, then the natural map
$$ H^{p}(Y, Gr^{F}_{q}DR(\cM(\ast E)) \otimes f^{\ast}\cV) \rightarrow  H^{p}(U, Gr^{F}_{q}DR(\cM) \otimes j^{\ast}\cV)$$
is an isomorphism for $p \leq -n_{S} -2$ and injective for $p = -n_{S} -1$.
\end{cor}

\begin{proof}
This follows from the previous proposition and the long exact sequence
$$
\cdots \rightarrow H^{p}_{E}(Y, Gr^{F}_{q}DR(\cM(\ast E)) \otimes f^{\ast}\cV) \rightarrow  H^{p}(Y, Gr^{F}_{q}DR(\cM(\ast E)) \otimes f^{\ast}\cV) -$$ 
$$\rightarrow  H^{p}(V, Gr^{F}_{q}DR(\cM) \otimes \rho^{\ast}f^{\ast}\cV) \rightarrow H^{p+1}_{E}(Y, Gr^{F}_{q}DR(\cM(\ast E)) \otimes f^{\ast}\cV) \rightarrow \cdots.$$

\end{proof}

Combining Theorem  \ref{thm2.2} and Corollary \ref{cor2.7}, we obtain a Kodaira-Saito type vanishing theorem for quasi-projective varieties.

\begin{thm}\label{thm2.9}
If $\cL$ is an ample invertible sheaf on $X$, then
$$H^{p}(U, Gr^{F}_{q}DR(\cM) \otimes j^{\ast}\cL^{-1}) = 0 \quad \text{for $p \leq -n_{S} -2$}.$$

\end{thm}

If we let $\cM = \Q^{H}_{U}[n]$ in the previous theorem, then we obtain Arapura's result.
\begin{cor}\cite{arapura}
 If $\cL$ is an ample invertible sheaf on $X$, then 
$$H^{p}(U, \Omega^{q}_{U} \otimes j^{\ast}\cL^{-1}) = 0\quad \text{for $p +q \leq n -n_{S} -2$}.$$
\end{cor}

\section{Pure Hodge Structures on Quasi-Projective Varieties}

We will keep the same setting as \ref{set1}, and $X$ will be an irreducible projective variety. We will assume the map $f:Y \rightarrow X$ is a strong log resolution of $S$. We will denote $n= n_{X}=n_{Y}$ and we will use the convention $\dim S = -1$ if $S = \emptyset$. 

\begin{lemma}\label{lemma3.1}
Let $Z$ be a smooth irreducible variety of dimension $m$ and $D$ a simple normal crossing divisor on $Z$. If $\cN \in HM_{Z}(Z,w)$ is a variation of Hodge structure $ ((H, F^{\bullet}), L)$ on $Z$, then there is a quasi-isomorphism
$$Gr^{F}_{p}DR(\varrho_{!}\varrho^{\ast}\cN) \simeq Gr^{F}_{p}DR(\varrho_{+}\varrho^{\ast}\cN)\otimes \cO_{Z}(-D),$$
where $\varrho:Z \backslash D \hookrightarrow Z$ is the natural map. 
\end{lemma}

\begin{proof}

The de Rham complex for $\varrho_{+}\varrho^{\ast}\cN$ is given by
$$(\Omega^{\bullet}_{Z}(\log D), F_{\bullet}) \otimes (H, F_{\bullet})[n],$$
where the filtration on $\Omega^{\bullet}_{Z}(\log D)$ is the ``stupid" filtration and the filtration on $H$ is given by $F_{p}H= F^{-p }H$ \cite[Prop. 3.11]{saito2}. Similarly, by \cite[Prop. 3.11]{saito2}, the de Rham complex of $\rho_{!}\rho^{*}\cN$ is given by
$$(\Omega^{\bullet}_{Z}(\log D), F_{\bullet}) \otimes (H \otimes \cO_{Z}(-D), F_{\bullet})[n].$$
The sheaf $\cO_{Z}(-D)$ is compatible with the filtration, and the lemma holds.

\end{proof}

\begin{prop}\label{prop3.2}
Let $\cM \in HM_{U}(U,w)$ be a variation of Hodge structure on $U$. If $D$ is a sufficiently general hyperplane of $X,$ and  $\epsilon: U \cap D \hookrightarrow U$ is the natural map, then the map
$$ H^{p}(U, Gr^{F}_{q}DR(\cM)) \rightarrow H^{p}(U \cap D, Gr^{F}_{q}DR(\epsilon^{\ast}\cM))$$
is bijective for $p < -n_{S} -2$ and injective for $p = -n_{S} -2$. 
\end{prop}

\begin{proof}

For general $D$,  $U \cap D$ is a smooth divisor on $U$. If $\varrho: U \backslash (U \cap D) \hookrightarrow U$ is the natural map, then by the previous lemma and Theorem \ref{thm2.9},
$$ H^{p}(U, Gr^{p}_{q}DR(\varrho_{!}\varrho^{\ast}\cM)) =  H^{p}(U, Gr^{p}_{q}DR(\varrho_{+}\varrho^{\ast}\cM) \otimes \cO_{U}(-D|_{U}))) = 0 \quad \text{whenever $p \leq -n_{S} -2$}.$$
 The proposition follows from the distinguished triangle 
$$\begin{tikzcd}
\varrho_{!}\varrho^{\ast}\cM \arrow[r] & \cM \arrow[r] & \epsilon_{+}\epsilon^{\ast}\cM \arrow[r, "+1"] & \hfill
\end{tikzcd}$$
and the resulting long exact sequence
$$\cdots \rightarrow H^{p}(U, Gr^{p}_{q}DR(\varrho_{!}\varrho^{\ast}\cM)) 
\rightarrow H^{p}(U,Gr^{F}_{q}DR(\cM))-$$
$$ \rightarrow H^{p}(U \cap D, Gr^{F}_{q}DR(\epsilon^{*}\cM)) \rightarrow  H^{p+1}(U, Gr^{p}_{q}DR(\varrho_{!}\varrho^{\ast}\cM)) \rightarrow \cdots.$$ 
\end{proof}

Let $\cM \in HM_{U}(U,w)$ be a variation of Hodge structure $((H, F^{\bullet}),L)$ on $U$.  Using the isomorphism $f\vert_{V}: V \rightarrow U,$ let's regard $\cM$ as a Hodge module on $V$ for the moment. Let $\cH^{\geq 0}_{Y}$ be the lattice of Deligne's regular singular meromorphic extension of $H$ such that the eigenvalues of $res \nabla$ along the irreducible components $E_{i}$ ($1 \leq i \leq r$) are contained in $[0, 1).$ The complex
$$ (\Omega^{\bullet}_{Y}(\log E), F_{\bullet}) \otimes (\cH^{\geq 0}_{Y},F_{\bullet})[n]$$ 
gives a filtered resolution of $DR(\cM(\ast E)) \simeq \rho_{\ast}DR(\cM)$ \cite[Prop 3.11]{saito2}. From Saito's original paper \cite{saito}, the Hodge filtration on $H^{i}(U, DR(\cM))$ is given by 
$$F_{p}^{H}H^{i}(U,DR(\cM))= Im[H^{i}(Y,F_{p}DR(\cM(\ast E))) \rightarrow H^{i}(U,DR(\cM))].$$
Since $Y$ is projective, the filtration is strict, and we obtain
$$Gr_{-p}^{F}H^{i}(U,DR(\cM)) = H^{i }(Y,Gr^{F}_{-p}DR(\cM(\ast E))) $$
and
$$H^{i}(U,DR(\cM)) = \bigoplus_{p \in \mathbb{Z}} H^{i }(Y, Gr^{F}_{p}DR( \cM(\ast E))).$$
In particular we have $F^{H}_{-p}H^{i}(U,DR(\cM)) = \displaystyle \bigoplus_{l \leq -p } H^{i }(Y, Gr^{F}_{l}DR(\cM(\ast E))).$ We may also define a second filtration on $H^{i}(U,DR(\cM))$ by
$$F_{p}^{DR}H^{i}(U, DR(\cM)):= Im[H^{i}(U, F_{p}DR(\cM)) \rightarrow H^{i}(U,DR(\cM))].$$
This filtration gives us a spectral sequence $E^{p,q}_{1} = H^{p +q}(U, Gr^{F}_{-p}DR(\cM))\Rightarrow H^{p +q}(U,DR(\cM)).$ However, this filtration may not be strict. Therefore, $E^{p,q}_{\infty} \cong Gr^{F}_{-p}H^{p+q}(U, DR(\cM)).$ The next proposition describes the relationship between the two filtrations. 

\begin{prop}\label{prop3.3}
For any $i$, $F^{H}_{p}H^{i}(U,DR(\cM)) \subseteq F^{DR}_{p}H^{i}(U,DR(\cM))$. If $i \leq  -n_{S} -2$ then:
\begin{enumerate}
\item These filtrations on $H^{i}(U,DR(\cM))$ coincide. 
\item For any $p$, $E^{p,i-p}_{1} = E^{p, i-p}_{\infty}$. 
\end{enumerate}
\end{prop}

\begin{proof}
   The natural map of complexes
$$(\Omega^{\bullet}_{Y}(\log E), F_{\bullet}) \otimes (\cH^{\geq 0}_{Y},F_{\bullet})[n] \rightarrow \bR \rho_{\ast}\big( (\Omega^{\bullet}_{V}, F_{\bullet}) \otimes (H,F_{\bullet})[n]\big) \simeq \rho_{\ast}\big( (\Omega^{\bullet}_{V}, F_{\bullet}) \otimes (H,F_{\bullet})[n]\big)$$
is a quasi-isomorphism since both complexes give a resolution of $\rho_{*}(L[n] \otimes \mathbb{C}) \simeq \rho_{*}DR(\cM)$. The map preserves the filtrations and induces an inclusion $$F^{H}_{p}H^{i}(U,DR(\cM)) \subseteq F^{DR}_{p}H^{i}(U,DR(\cM)).$$
If $i \leq  -n_{S} -2$, then because of Corollary \ref{cor2.7},
$$F^{H}_{-p}H^{i}(U,DR(\cM)) = \bigoplus_{l \leq -p } H^{i }(Y, Gr^{F}_{l}DR(\cM(\ast E))) \cong \bigoplus_{l \leq -p }H^{i}(U, Gr^{F}_{l}DR(\cM)).$$
 Since the spaces $E^{p, i-p}_{\infty}$ are subquotients of $H^{i}(U, Gr^{F}_{-p}DR(\cM)),$ there is an inequality $\dim F^{DR}_{-p}H^{i}(U,DR(\cM)) \leq \dim F^{H}_{-p} H^{i}(U,DR(\cM)).$ So we must have $$F^{H}_{-p}H^{i}(U,DR(\cM)) = F^{DR}_{-p}H^{i}(U,DR(\cM))\quad  \text{and} \quad  E^{p,i-p}_{1} = E^{p, i-p}_{\infty}.$$

\end{proof}

\begin{rmk}
    In the previous discussion and proposition, we were working in the analytic category. This is because the perverse sheaf $ DR(\cM) :=  DR(\cM^{an})$ in defined on $U^{an}$, where $U^{an}$ is the associated analytic space. However, since we are working with projective varieties, we may apply GAGA \cite{gaga} when considering the Hodge filtration,
    $$H^{i }(Y^{an}, Gr^{F}_{p}DR( \cM^{an}(\ast E))) \cong H^{i }(Y, Gr^{F}_{p}DR( \cM(\ast E))).$$
    Also, by applying a result of Siu \cite[Thm. B]{Siu}, for a variation of Hodge structure $\cM \in HM_{U}(U,w)$, 
$$ H^{i}(U^{an},Gr^{F}_{-p}DR(\cM^{an})) \cong H^{i}(U,Gr^{F}_{-p}DR(\cM)) \quad \text{if $i \leq -n_{S} -2$}.$$
The previous isomorphism is discussed in more detail in the last section of the paper. 
\end{rmk}

\begin{prop}\label{prop3.4}
Suppose $\cM \in HM_{U}(U,w)$ is a variation of Hodge structure. If $D$ is a sufficiently general hyperplane section of $X$, and  $\epsilon: U \cap D \hookrightarrow U$ is the natural map, then the map
$$H^{i-n}(U,DR(\cM)) \rightarrow H^{i-n}(U \cap D, DR(\epsilon^{\ast}\cM))$$
is an isomorphism for $i < n - n_{S} -2$ and injective for $i = n - n_{S} -2$.
\end{prop}

\begin{proof}
For general $D$, $U \cap D$ is smooth and $codim(S \cap D, D) \geq codim(S,X) -1$. Equality holds only if $S = \emptyset$. The result follows from Propositions \ref{prop3.2} and \ref{prop3.3}.
\end{proof}

\begin{thm}\label{thm3.5}
  If $\cM \in HM_{U}(U,w)$ is a variation of Hodge structure, then the mixed Hodge structure on $H^{i-n}(U,DR(\cM))$ is pure of weight $w +i -n$ provided that $i \leq n- n_{S} -2$. Furthermore, 
$$H^{i-n}(U,DR(\cM)) \cong \bigoplus_{p \in \Z} H^{i - n}(U, Gr^{F}_{p}DR(\cM))$$ 
and
$$dim[H^{i-n}(U, Gr^{F}_{-p}DR(\cM))] = dim[H^{ i -n}(U, Gr^{F}_{p +n -w -i}DR(\cM))] .$$

\end{thm}

\begin{proof}
The theorem holds if $S = \emptyset$ because $U =X$ is a smooth projective variety. If $S \neq \emptyset$, then for a sufficiently general hyperplane $D$, the intersection $U \cap D$ is smooth and $codim(S \cap D, D) = codim(S, X)$. If $i \leq n - n_{S} -2$, then $H^{i-n}(U,DR(\cM))$ injects into $H^{i-n}(U \cap D, DR(\epsilon^{*}\cM))$ by the previous proposition. By induction on the dimension of $U$, the Hodge structure $H^{i-n}(U \cap D, DR(\epsilon^{*}\cM))$ is pure of weight $w +i -n$. Therefore, $H^{i-n}(U,DR(\cM))$ is also pure of weight $w +i -n$. The decomposition follows from Proposition \ref{prop3.3}.
\end{proof}

\subsection{Applications}

\begin{thm}\label{Hodge}
     Let $\bH= ((H,F^{\bullet}),L)$ be a variation of Hodge structure of weight $w-n$ on $U$, and $\cM \in HM_{U}(U,w)$ the corresponding Hodge module. Let $IC^{H}_{X}(L) \in HM_{X}(X,w)$ be the unique Hodge module on $X$ that extends $\bH,$ whose underlying perverse sheaf is the intersection complex $IC_{X}(L)$ twisted by $L$. The natural map 
     $$H^{i-n}(X,IC_{X}(L)) \rightarrow H^{i}(U,L)$$ 
     is an isomorphism of pure $\mathbb{Q}$-Hodge structures of weight $w +i -n$ whenever $i \leq n- n_{S} -1.$ Furthermore, if $i \leq n -n_{S} -2,$ then the natural map
     $$H^{i-n}(X, Gr^{F}_{-p}DR(IC^{H}_{X}(L))) \rightarrow H^{i-n}(U, Gr^{F}_{-p}DR(\cM))$$
     is an isomorphism. 
 \end{thm}

 \begin{proof}
    By the assumption, we have $\cM = j^{*}IC^{H}_{X}(L)$. There is a natural triangle in $D^{b}MHM(X)$ given by 
    $$\begin{tikzcd}
        i_{S+}i^{!}_{S}IC^{H}_{X}(L) \arrow[r] & IC^{H}_{X}(L) \arrow[r] & j_{+}j^{*}IC^{H}_{X}(L) \arrow[r, "+1"] & \hfill.
    \end{tikzcd}$$
   The previous exact triangle of mixed Hodge modules is compatible with the following exact triangle of perverse sheaves,
    $$\begin{tikzcd}
        i_{S*}i^{!}_{S}IC_{X}(L) \arrow[r] & IC_{X}(L) \arrow[r] & \bR j_{*}L[n] \arrow[r, "+1"] & \hfill.
    \end{tikzcd}$$
    By applying the derived functor $\bR \Gamma(X, \bullet),$ and taking cohomology, then we receive the natural map
    $$H^{i-n}(X, IC_{X}(L)) \rightarrow H^{i}(U, L).$$
    Using the same argument given by \cite[Thm. 6.7.4]{max}, the natural map $H^{i-n}(X, IC_{X}(L)) \rightarrow H^{i}(U, L)$ is an isomoprhism whenever $i < n- n_{S}.$ Moreover, this isomorphism must be an isomorphism of Hodge structures. The theorem now follows from Theorem \ref{thm3.5}.
 \end{proof}

\begin{prop}\label{cor3.6}
Assume $\cM \in HM_{U}(U,w)$ is a variation of Hodge structure and let $\cM_{Y} \in HM_{Y}(Y,w)$ be the unique Hodge module that extends $f\vert_{V}^{*}\cM$. If $i \leq n - n_{S} -2$, then the natural map 
$$H^{i-n}(Y, Gr^{F}_{p}DR(\cM_{Y})) \rightarrow H^{i-n}(U, Gr^{F}_{p}DR(\cM))$$
is surjective. 
\end{prop}

\begin{proof}
We will abuse notation by denoting $f\vert_{V}^{*}\cM$ as simply $\cM.$ Since $E$ is a divisor on $Y$ and $\cM_{Y}$ has no non-trivial sub-object supported on $E$, there is a short exact sequence of mixed Hodge modules
$$0 \rightarrow \cM_{Y} \rightarrow \cM(*E) \rightarrow \iota_{+}\cH^{1}\iota^{!}\cM_{Y} \rightarrow 0.$$
For notation let $\cH^{1}_{E}(\cM_{Y}):=\iota_{+}\cH^{1}\iota^{!}\cM_{Y}$. If $\pi:Y \rightarrow \{pt\}$, then there is an exact triangle
$$\begin{tikzcd}
\pi_{+}\cM_{Y} \arrow[r] & \pi_{+}\cM(*E) \arrow[r] & \pi_{+} \cH^{1}_{E}(\cM_{Y}) \arrow[r, "+1"] & \hfill.
\end{tikzcd}$$
By taking the cohomology of the exact triangle above, we have the resulting log exact sequence of mixed Hodge structures
$$ \cdots \rightarrow H^{i-n}\pi_{+}\cM_{Y} \rightarrow H^{i-n}\pi_{+}\cM(*E) \rightarrow H^{i-n}\pi_{+}\cH^{1}_{E}(\cM_{Y}) \rightarrow H^{i +1 -n}\pi_{+}\cM_{Y} \rightarrow \cdots.$$
By the direct image theorem of mixed Hodge modules \cite[Thm. 4.13]{ks}, there is a convergent weight spectral sequence for $\pi_{+}\cH^{1}_{E}(\cM_{Y})$ given by
$$E^{p,q}_{1}= H^{p +q}\pi_{+}Gr^{W}_{-p}\cH^{1}_{E}(\cM_{Y}) \Rightarrow H^{p +q}\pi_{+}\cH^{1}_{E}(\cM_{Y}),$$
and each differential $d_{1}:E^{p,q}_{1} \rightarrow E^{p+1,q}_{1}$ is a morphism in $HM(pt,q)$. The weight spectral sequence degenerates at $E_{2}$, and we have
$$Gr^{W}_{q}H^{p+q}\pi_{+}\cH^{1}_{E}(\cM_{Y}) \cong E^{p,q}_{2} \quad \text{for every $p,q \in \mathbb{Z}$}.$$
But, by \cite[Prop. 2.26]{saito2}, $Gr^{W}_{p}\cH^{1}_{E}(\cM_{Y}) = 0$ for $p < w +1$. Therefore,
$$E^{-w, w +i -n}_{1} = H^{i-n}Gr^{W}_{w}\pi_{+}\cH^{1}_{E} (\cM_{Y}) = 0$$
and
$$E^{-w, w +i -n}_{2} \cong Gr^{W}_{w +i -n}H^{i -n}\pi_{+}\cH^{1}_{E}(\cM_{Y})= 0.$$
The Hodge structure $H^{i-n}\pi_{+}\cM_{Y} = H^{i-n}(Y,DR(\cM_{Y}))$ is pure of weight $w +i -n$ by the direct image theorem \cite{saito}. If $i \leq n -n_{S} -2$, then by Theorem \ref{thm3.5}, the Hodge structure $H^{i-n}\pi_{+}\cM(*E) = H^{i-n}(U, DR(\cM))$ is pure of weight $w +i -n$, too. The functor $Gr^{W}_{w +i -n}(\bullet)$ is exact, and we have shown that $Gr^{W}_{w +i -n}H^{i -n}\pi_{+}\cH^{1}_{E}(\cM_{Y})= 0.$ So, by the long exact sequence above, there is a short exact sequence
$$0 \rightarrow Gr^{W}_{w +i -n}H^{i-n -1}\pi_{+}\cH^{1}_{E}(\cM_{Y}) \rightarrow H^{i-n}\pi_{+}\cM_{Y} \rightarrow H^{i-n}\pi_{+}\cM(*E) \rightarrow 0.$$
If we apply the functor $Gr^{F}_{p}(\bullet)$ to the exact sequence above, there is a surjective map
$$H^{i-n}(Y,Gr^{F}_{p}DR(\cM_{Y})) \cong Gr^{F}_{p}H^{i-n}\pi_{+}\cM_{Y}  \rightarrow  Gr^{F}_{p}H^{i-n}\pi_{+}\cM(*E) \cong H^{i-n}(U,Gr^{F}_{p}DR(\cM)) .$$
\end{proof}

\begin{cor}
    If $i+p \leq n - n_{S} -2,$ then the natural map 
    $$H^{i}(Y, \Omega^{p}_{Y}) \rightarrow H^{i}(U, \Omega^{p}_{U})$$
is surjective. 
\end{cor}
\begin{proof}
    Let $\cM$ be the trivial Hodge module $ \mathbb{Q}^{H}_{U}[n]$. Using the same notation as in Proposition \ref{cor3.6}, we have $\cM_{Y}= \mathbb{Q}^{H}_{Y}[n]$. If $i \leq n-n_{S} -2,$ then for $1 \leq p \leq n$ the natural map
    $$H^{i-n}(Y, \Omega^{p}_{Y}[n-p])  \rightarrow  H^{i-n}(U, \Omega^{p}_{U}[n-p])$$
    is surjective.
\end{proof}

\begin{prop}\label{cor3.7}
Assume $\cM \in HM_{U}(U,w)$ is a variation of Hodge structure and let $\cM_{Y} \in HM_{Y}(Y,w)$ be the unique Hodge module that extends $f\vert_{V}^{*}\cM$. Let 
$$p(\cM_{Y}) = \min\{p: F_{p}M_{Y} \neq 0\},$$ 
where $M_{Y}$ is the underlying right $\cD_{Y}$-module of $\cM_{Y}$. If $i \leq n - n_{S} -2$, then 
$$  H^{i -n}(Y, Gr^{F}_{p(\cM_{Y}) +n - i}DR(\cM_{Y})) \cong  H^{i-n}(U, Gr^{F}_{p(\cM_{Y}) + n - i}DR(\cM)).$$
 \end{prop}

\begin{proof}
From Proposition \ref{cor3.6} it suffices to show
$$ \dim H^{i -n}(Y, Gr^{F}_{p(\cM_{Y}) +n - i}DR(\cM_{Y})) = \dim H^{i-n}(U, Gr^{F}_{p(\cM_{Y}) + n - i}DR(\cM)).$$
Since $\cM_{Y}$ is a pure polarized Hodge module and $Y$ is projective, there is equality
$$\dim  H^{i -n}(Y, Gr^{F}_{p(\cM_{Y}) +n - i}DR(\cM_{Y})) =  \dim H^{i -n}(Y, Gr^{F}_{-p(\cM_{Y}) -w}DR(\cM_{Y})).$$
Again, since $E$ is a divisor on $Y$ and $\cM_{Y}$ has no non-trivial sub-object supported on $E$, there is a short exact sequence of mixed Hodge modules
$$0 \rightarrow \cM_{Y} \rightarrow \cM(*E) \rightarrow \iota_{+}\cH^{1}\iota^{!}\cM_{Y} \rightarrow 0.$$
Let $\cH^{1}_{E}(\cM_{Y}):=\iota_{+}\cH^{1}\iota^{!}\cM_{Y}$. We have  $Gr^{F}_{-p(\cM_{Y}) - w}DR(\cH^{1}_{E}(\cM_{Y})) \simeq 0$  because $Gr^{W}_{p}(\cH^{1}_{E}(\cM_{Y})) = 0$ for $p < w +1$ and $p(\cM_{Y}) \leq \min\{p: F_{p}\cH^{1}_{E}(M_{Y})) \neq 0\}$ (see \cite[Prop. 4.7]{ks} or \cite[Prop. 1.8]{AH} for more details).  Hence there is a quasi-isomorphism
$$Gr^{F}_{-p(\cM_{Y}) - w}DR(\cM_{Y}) \simeq Gr^{F}_{-p(\cM_{Y}) -w}DR(\cM(\ast E)),$$
and
$$\dim  H^{i -n}(Y, Gr^{F}_{p(\cM_{Y}) +n - i}DR(\cM_{Y})) =  \dim H^{i -n}(Y, Gr^{F}_{-p(\cM_{Y}) -w}DR(\cM(\ast E))) .$$
By Corollary \ref{cor2.7}, we have
$$ \dim H^{i - n}(Y, Gr^{F}_{-p(\cM_{Y})-w}DR(\cM(\ast E))) = \dim H^{i- n}(U, Gr^{F}_{-p(\cM_{Y}) - w}DR(\cM)).$$
By Theorem \ref{thm3.5}, we have the equality
$$ \dim H^{i - n}(U, Gr^{F}_{p(\cM_{Y}) +n - i}DR(\cM)) = \dim H^{i- n}(U, Gr^{F}_{-p(\cM_{Y}) - w}DR(\cM)).$$
Putting all this together, we have the desired result. 

\end{proof}

\begin{cor}\label{cor3.8}\cite{arapura}
If $i \leq n - n_{S} -2,$ then natural map
$$\Gamma(Y, \Omega^{i}_{Y}) \rightarrow \Gamma(U, \Omega^{i}_{U})$$
is an isomorphism.
\end{cor}

\begin{proof}
Let $\cM = \mathbb{Q}^{H}_{U}[n]$ be the trivial Hodge module. Using the same notation as in Proposition \ref{cor3.6}, $\cM_{Y} = \mathbb{Q}^{H}_{Y}[n]$  and  $p(\cM_{Y}) = -n$. Provided $i \leq n - n_{S} -2$, the previous proposition gives us 
$$ H^{0}(Y, \Omega^{i}_{Y}) = H^{i-n}(Y,Gr^{F}_{-i}DR(\mathbb{Q}^{H}_{Y}[n])) \cong   H^{i-n}(U,Gr^{F}_{-i}DR(\mathbb{Q}^{H}_{U}[n]))=  H^{0}(U, \Omega^{i}_{U}).$$
 \end{proof}

\section{A Local Vanishing Theorem in Local Cohomology}

Let $X$ and $Y$ be algebraic varieties satisfying the conditions in Setting \ref{set1}. Unless stated otherwise, assume for this section that $f: Y \rightarrow X$ is a proper birational map with $X$ projective and equidimensional. Also, assume $U$ is nonsingular and $f\vert_{V}: V \rightarrow U$ is an isomorphism. Since $n_{Y} = n_{X}$, we will let $n = n_{Y} = n_{X}.$ Recall that in the case $f: Y \rightarrow X$ is a resolution of singularities, we will always assume $S$ contains the singular locus of $X,$ and $S$ does not contain any irreducible component of $X.$ 

\begin{thm} \label{prop1.9} Let $\cF \in D^{b}_{coh}(\cO_{Y})$ and assume $\cF$ is quasi-isomorphic to a complex $\tcF$ with the following properties: 
\begin{enumerate}
\item  $\sigma_{\leq -n -1}\tcF = 0$ and there exists $l \in [0 ,n]$ such that $ \sigma_{> -n+l}\tcF = 0$ \\

\item There exists $q \geq - n + l +1$ such that $R^{i}f_{\ast}\tcF= 0$ for all $i \geq q$ \\

\item  For all $j \in \mathbb{Z},$ the sheaf $\tcF^{j}\vert_{V}$ is locally free on $V$.
\end{enumerate} 
Then  
$$R^{p} f_{\ast}\bR \cH om_{E}(\cF, \omega^{\bullet}_{Y}) = 0 \quad \text{for $p \leq -n_{S} - q$}.$$
\end{thm}

\begin{proof}
Since $R^{p}f_{\ast}\bR \cH om_{E}(\cF, \omega^{\bullet}_{Y}) \cong R^{p} f_{\ast}\bR \cH om_{E}(\tcF, \omega^{\bullet}_{Y})$, we may assume $\cF = \tcF$.  \\

$\underline{\text{Step 1}}$: We first show $R^{p} f_{\ast}\bR \cH om_{E}(\cF, \omega^{\bullet}_{Y})$ is a coherent sheaf on $X$ for $p \leq -n_{S} -q$. 
We have a distinguished triangle
$$\begin{tikzcd}
\bR \cH om_{E}(\cF, \omega^{\bullet}_{Y}) \arrow[r] & \bR \cH om_{Y}(\cF, \omega^{\bullet}_{Y}) \arrow[r] & \bR \rho_{\ast} \rho^{\ast}\bR \cH om_{Y}(\cF, \omega^{\bullet}_{Y}) \arrow[r, "+1"] & \hfill.
\end{tikzcd}$$
Applying the right derived functor $\bR f_{\ast}$ to this triangle we obtain the exact triangle
$$\begin{tikzcd}
\bR f_{\ast}\bR \cH om_{E}(\cF, \omega^{\bullet}_{Y}) \arrow[r] & \bR f_{\ast}\bR \cH om_{Y}(\cF, \omega^{\bullet}_{Y}) \arrow[r] & \bR f_{\ast}\bR \rho_{\ast} \rho^{\ast}\bR \cH om_{Y}(\cF, \omega^{\bullet}_{Y}) \arrow[r, "+1"] & \hfil.
\end{tikzcd}$$
Since  $f\vert_{V}:V \rightarrow U$ is an isomorphism, we obtain the quasi-isomorphisms
$$\bR f_{\ast}\bR \rho_{\ast} \rho^{\ast}\bR \cH om_{Y}(\cF, \omega^{\bullet}_{Y}) \simeq \bR f_{\ast} \bR \rho_{\ast} \bR \cH om_{V}(\cF\vert_{V}, \omega^{\bullet}_{V}) \simeq \bR j_{\ast} f\vert_{V \ast}\bR \cH om_{V}(\cF\vert_{V}, \omega^{\bullet}_{V}).$$
For notation, let $ \cG := f\vert_{V \ast}\bR \cH om_{V}(\cF\vert_{V}, \omega^{\bullet}_{V}) \in D^{b}_{coh}(\cO_{U})$. Then, there is a distinguished triangle
$$
\begin{tikzcd}
\bR f_{\ast}\bR \cH om_{E}(\cF, \omega^{\bullet}_{Y}) \arrow[r] & \bR f_{\ast}\bR \cH om_{Y}(\cF, \omega^{\bullet}_{Y}) \arrow[r] & \bR j_{\ast}\cG \arrow[r, "+1"] & \hfill.
\end{tikzcd}$$
Since $\cF \in D^{b}_{coh}(\cO_{Y})$ and $f:Y \rightarrow X$ is proper,  $ \bR f_{\ast}\bR \cH om_{Y}(\cF, \omega^{\bullet}_{Y}) \in D^{b}_{coh}(\cO_{X})$. Therefore, to show $R^{p} f_{\ast}\bR \cH om_{E}(\cF, \omega^{\bullet}_{Y})$ is coherent on $X$ for $p \leq -n_{S} - q $, it suffices to show $R^{p}j_{\ast}\cG$ is coherent for $p \leq - n_{S} - q -1 $.

By our assumptions, the complex $\cF$ is equal to
$$\cdots \rightarrow 0 \rightarrow \cF^{-n} \rightarrow \cF^{-n+1} \rightarrow \cdots \rightarrow \cF^{ -n+l} \rightarrow 0 \rightarrow \cdots.$$
 Since $V$ is a smooth  equidimensional variety and $\cF^{j}$ is locally free, the complex $\bR \cH om_{V}(\cF\vert_{V}, \omega^{\bullet}_{V})$ is given by
$$\cdots \rightarrow 0 \rightarrow (\cF^{-n+l}\vert_{V})^{\vee} \otimes \Omega^{n}_{V} \rightarrow (\cF^{ -n +l -1}\vert_{V})^{\vee}\otimes \Omega^{n}_{V} \rightarrow \cdots \rightarrow (\cF^{-n}\vert_{V})^{\vee} \otimes \Omega^{n}_{V} \rightarrow 0 \rightarrow \cdots.$$
The sheaf $(\cF^{j}\vert_{V})^{\vee} \otimes \Omega^{n}_{V}$ is in the $-(j +n)^{th}$ spot for $j \in [-n, -n +l]$. Pushing forward this complex onto $U$ by $f\vert_{V \ast}$, we obtain the complex $\cG$. By \cite[Lemma I.7.2]{Hart1}, there is an exact sequence of complexes given by
$$ 0 \rightarrow \sigma_{> -l}\cG \rightarrow \cG \rightarrow \cG^{-l}[l] \rightarrow 0.$$
By applying $\bR j_{\ast}$ to the exact sequence above, there is an exact triangle
$$\begin{tikzcd}
 \bR j_{\ast}\sigma_{> -l}\cG \arrow[r] & \bR j_{\ast}\cG \arrow[r]& \bR j_{\ast}\cG^{-l}[l] \arrow[r, "+1"] & \hfill. 
 \end{tikzcd}$$
First, we show $R^{p}j_{\ast}\cG^{-l}[l]$ is coherent for $p \leq - n_{S} -q -1.$ Which is equivalent to showing $R^{p}j_{\ast}\cG^{-l}$ is coherent for $p \leq  l- n_{S}-q  - 1.$ Note that we have the inequality 
$$p \leq  l- n_{S}-q  - 1 \leq n -n_{S} -2$$ 
because of the second assumption of this proposition. The sheaf $\cG^{-l} $ is coherent on $U$. It is well known that there is a coherent sheaf $\cG'$ on $X$ such that $\cG'\vert_{U} \cong \cG^{-l}$ \cite[II. 5]{Hart4}. Therefore, there exists a distinguished triangle 
$$\begin{tikzcd}
\bR \underline{\Gamma_{S}}(\cG') \arrow[r] & \cG' \arrow[r] & \bR j_{\ast}\cG^{-l} \arrow[r, "+1"] & \hfill.
\end{tikzcd}$$
The local cohomology sheaves $\cH^{i}_{S}(\cG')$ are coherent for $i < n -n_{S}$ because $\cG^{-l}$ is locally free on the smooth variety $U$ \cite{Siu} (see also \cite[VI. vi]{Hart3}). Therefore, the sheaf $R^{p}j_{\ast}\cG^{-l}$ is coherent for $p \leq n - n_{S} -2$. Hence $R^{p}j_{\ast}\cG^{-l}$ is coherent for $p \leq l -n_{S} -q -1$. Now by induction on $l$, $R^{p}j_{*}\sigma_{>-l}\cG$ is coherent for $p \leq -n_{S} - q -1$. Hence $R^{p}j_{*}\cG$ is coherent for $p \leq -n_{S} - q -1$.\\

$\underline{\text{Step 2}}$: Let $\cL$ be a very ample sheaf on $X$. There is a spectral sequence
$$E^{m,p}_{2}= H^{m}(X, R^{p}f_{\ast}\bR \cH om_{E}(\cF, \omega^{\bullet}_{Y}) \otimes \cL^{N}) \Rightarrow Ext^{p +m}_{E}(\cF \otimes f^{\ast}\cL^{-N}, \omega^{\bullet}_{Y}).$$
By Step 1, $R^{p} f_{\ast}\bR \cH om_{E}(\cF, \omega^{\bullet}_{Y})$ is coherent for $p \leq -n_{S} -q$. If $p \leq -n_{S} -q$ and $N \gg 0$, then 
$$E^{m,p}_{2} = H^{m}(X, R^{p} f_{\ast}\bR \cH om_{E}(\cF, \omega^{\bullet}_{Y}) \otimes \cL^{N}) = 0 \quad \text{for $m > 0$}.$$
Note that $E^{m,p}_{2} = 0$ whenever $m<0$. So, for $N \gg 0$ and  $p \leq -n_{S} -q$, we have 
$$E^{0, p}_{2} =\Gamma(X, R^{p}f_{\ast}\bR \cH om_{E}(\cF, \omega^{\bullet}_{Y}) \otimes \cL^{N})  \cong  Ext^{p}_{E}(\cF \otimes f^{\ast}\cL^{-N}, \omega^{\bullet}_{Y}).  $$
By Theorem \ref{thm1.5},  $Ext^{p}_{E}(\cF \otimes f^{\ast}\cL^{-N}, \omega^{\bullet}_{Y}) = 0$ for $p \leq -n_{S} - q$. Therefore, since the sheaf 
$R^{p} f_{\ast}\bR \cH om_{E}(\cF, \omega^{\bullet}_{Y}) \otimes \cL^{N}$ is globally generated for $N \gg 0$, we can conclude $$R^{p}f_{\ast}\bR \cH om_{E}(\cF, \omega^{\bullet}_{Y})  = 0 \quad  \text{for $p \leq -n_{S} -q$}.$$ 

\end{proof}

\begin{cor}\label{cor1.10}
With the same assumptions as Theorem \ref{prop1.9}, the natural map
$$ R^{p} f_{\ast}\bR \cH om_{Y}(\cF, \omega^{\bullet}_{Y})\rightarrow R^{p}j_{\ast}\bR \cH om_{U}(f\vert_{V \ast}\cF\vert_{V}, \Omega^{n}_{U}[n])$$
is an isomorphism for $p \leq -n_{S} -q -1$ and injective for $p = -n_{S} - q$.
\end{cor}

\begin{proof}
This follows from the previous theorem and the exact triangle
$$\begin{tikzcd}
\bR f_{\ast}\bR \cH om_{E}(\cF, \omega^{\bullet}_{Y}) \arrow[r] & \bR f_{\ast}\bR \cH om_{Y}(\cF, \omega^{\bullet}_{Y}) \arrow[r] & \bR f_{\ast}\bR \rho_{\ast} \rho^{\ast}\bR \cH om_{Y}(\cF, \omega^{\bullet}_{Y}) \arrow[r, "+1"] & \hfil.
\end{tikzcd}$$
\end{proof}

\subsection{Higher Direct Images of Log Forms for Projective Varieties} Assume $f:Y \rightarrow X$ is a strong log resolution of $S$.
Let $\bH= ((\cL, F^{\bullet}), L)$ be a variation of Hodge structure on $V$ and let $\cM \in HM_{V}(V,w)$ be the corresponding Hodge module on $V.$ It was shown by Saito \cite[Prop. 3.11]{saito2} that the complex $Gr^{F}_{p}DR(\cM(!E))$ is quasi-isomorphic to the following complex, 
$$\tcF:=[0 \rightarrow \cO_{Y} \otimes Gr^{F}_{p}\tcL^{>0} \rightarrow \Omega^{1}_{Y}(\log E) \otimes Gr^{F}_{p+1}\tcL^{>0} \rightarrow \cdots \rightarrow \Omega^{n}_{Y}(E)\otimes Gr^{F}_{n +p}\tcL^{>0} \rightarrow 0][n],$$
where $\tcL^{>0}$ is the lattice of Deligne's regular singular meromorphic extension of $\cL$ such that the eigenvalues of $res(\nabla)$ along the irreducible components of $E$ are contained in $(0, 1].$ Note that 
$$\Omega^{i}_{Y}(\log E) \otimes Gr^{F}_{i +p}\tcL^{>0}$$
is in the $-(n-i)^{th}$-spot. Hence $$\sigma_{\leq -n-1}\tcF = 0 \quad \text{and}  \quad \sigma_{>0}\tcF = 0.$$ 
By Corollary \ref{cor2.3} $R^{i}f_{\ast}Gr^{F}_{p}DR(\cM(! E)) = 0 $ whenever $i \geq 1.$ Finally, for all $i\in \mathbb{Z}$ the sheaves $$(\Omega^{i}_{Y}(\log E) \otimes Gr^{F}_{i +p}\tcL^{>0})\vert_{V} \cong \Omega^{i}_{V} \otimes Gr_{F}^{-i -p}\cL$$
are locally free on $V$. Therefore, $\tcF$ satisfies the three conditions of Theorem \ref{prop1.9} with $l= n$ and $q =1.$

\begin{prop}\label{prop4.1}
With the assumptions above, the natural map
$$R^{i}f_{\ast}Gr^{F}_{-p}DR(\cM(\ast E)) \rightarrow R^{i}j_{\ast}f|_{V \ast}Gr^{F}_{-p }DR(\cM)$$

is an isomorphism for $i \leq -n_{S} -2$ and injective for $i = -n_{S} -1$. 
\end{prop}
 \begin{proof}
  By Corollary \ref{cor1.10} and the discussion above, the natural map
  $$ R^{i}f_{*}\bR \cH om_{Y}(Gr^{F}_{p}DR(\D(\cM)(!E), \omega^{\bullet}_{Y}) \rightarrow R^{i}j_{*}f\vert_{V*}\bR \cH om_{V}(Gr^{F}_{p}DR(\cM(w)), \omega^{\bullet}_{V})$$
  is an isomorphism for $i \leq -n_{S} -2$ and injective for $i = -n_{S}-1.$ Since $\cM$ is pure of weight $w,$ there is a quasi-isomorphism
  $$\bR \cH om_{V}(Gr^{F}_{p}DR(\cM(w)), \omega^{\bullet}_{V}) \simeq Gr^{F}_{-p}DR(\cM).$$ Furthermore, there is a quasi-isomorphism 
  $$\bR \cH om_{Y}(Gr^{F}_{p}DR(\D(\cM)(!E), \omega^{\bullet}_{Y}) \simeq Gr^{F}_{-p}DR(\cM(\ast E)).$$ Thus, the proposition holds from the following commutative diagram,
  $$\begin{tikzcd} R^{i}f_{*}\bR \cH om_{Y}(Gr^{F}_{p}DR(\D(\cM)(!E), \omega^{\bullet}_{Y}) \arrow[r, "\cong"] \arrow[d] & R^{i}f_{*}Gr^{F}_{-p}DR(\cM(\ast E)) \arrow[d] \\
   R^{i}j_{*}f\vert_{V*}\bR \cH om_{V}(Gr^{F}_{p}DR(\cM(w)), \omega^{\bullet}_{V})\arrow[r, "\cong"] & R^{i}j_{*}f\vert_{V*}Gr^{F}_{-p}DR(\cM). \end{tikzcd}$$
 \end{proof}

\begin{cor}\label{cor4.2}
 If $f:Y \rightarrow X$ is a strong log resolution of $S$, then the natural map
$$R^{i}f_{\ast}\Omega^{p}_{Y}(\log E) \rightarrow R^{i}j_{\ast}\Omega^{p}_{U}$$
is an isomorphism for $p +i \leq n - n_{S} -2$ and injective for $p +i = n- n_{S} - 1$. 
\end{cor}

\begin{proof}
If we let $\cM = \Q^{H}_{V}[n]$, then $Gr^{F}_{-p}DR(\cM(\ast E)) = \Omega^{p}_{Y}(\log E)[n-p]$ and $Gr^{F}_{-p}DR(\cM) = \Omega^{p}_{V}[n-p]$. Now, apply the previous proposition.  
\end{proof}

We need the following definition from \cite{Hart2} before getting into the following application. 

\begin{defn}
 Let $Z$ be a local Noetherian scheme, $\Delta$ a closed subset of $Z$, and $F$ a coherent sheaf on $Z$. Then $depth_{\Delta}F$ is the $\displaystyle \inf_{z \in \Delta}($depth $F_{z})$. 
\end{defn}

\begin{cor}\label{cor4.4}
Let $S$ be the singular locus of $X$. For any resolution of singularities $f:Y \rightarrow X,$ 
$$R^{i}f_{\ast}\cO_{Y} = 0 \quad \text{for $1 \leq i \leq depth_{S}\cO_{X} -2$}.$$
\end{cor}

\begin{proof}
The complex $\bR f_{\ast}\cO_{Y}$ is independent of the resolution. So, we may assume $f:Y \rightarrow X$ is a strong log resolution. By the previous corollary, the map $R^{i}f_{\ast}\cO_{Y} \rightarrow R^{i}j_{\ast}\cO_{U}$ is always injective for $i \leq n -n_{S} -1$. For $i\geq 1$ there is an isomorphism $R^{i}j_{\ast}\cO_{U} \cong \cH^{i+1}_{S}(\cO_{X})$. The corollary therefore holds since the sheaf $\cH^{i +1}_{S}(\cO_{X})$ is zero for $i +1 < depth_{S}\cO_{X}$ \cite[Thm. 3.8]{Hart2}.
\end{proof}

\begin{rmk}
If $\tilde{S}$ is the the smallest closed subset of $X$ such that $X \backslash \tilde{S}$ has rational singularities, then we can also show
$$R^{i}f_{\ast}\cO_{Y} = 0 \quad \text{for $1 \leq i \leq depth_{\tilde{S}}\cO_{X} -2$}.$$
This can be shown by applying Proposition \ref{prop1.9} and Corollary \ref{cor1.10}. Let $\tilde{V} = Y \backslash f^{-1}(\tilde{S})$ and $\tilde{U} = X \backslash \tilde{S}$ with the induced map $g: \tilde{V} \rightarrow \tilde{U}$. The map $g: \tilde{V} \rightarrow \tilde{U}$ may not be an isomorphism. Still, with a strong log resolution $f: Y \rightarrow X$ that is functorial with smooth morphisms, there is an isomorphism $\bR g_{\ast}\cO_{\tilde{V}} \simeq g_{\ast}\cO_{\tilde{V}} \simeq \cO_{\tilde{U}}$. Using this isomorphism, we can still use the proof Proposition \ref{prop1.9} and Corollary \ref{cor1.10}. In the case when $X$ is Cohen-Macaulay, we obtain
$$R^{i}f_{\ast}\cO_{Y} = 0 \quad \text{for $1 \leq i \leq n - n_{\tilde{S}} -2$}.$$
This is known even in the non-projective case. One can see \cite{kovacs} for a proof. 
\end{rmk}

\subsection{The Graded Pieces of the de Rham Complex and the Extension Problem}

\begin{lemma}
    Let $Z$ be a smooth variety and $T \subseteq Z$ a closed subvariety of dimension $n_{T}.$ If $\cM \in MHM(Z)$ whose support is contained in $T,$ then 
    $$\cH^{i}(Gr^{F}_{p}DR(\cM)) = 0 \quad \text{for $i < -n_{T}$}.$$
\end{lemma}

\begin{proof}
    By induction on the weight, we may assume $\cM$ is a pure Hodge module of weight $w.$ By induction on dimension of $T,$ using the decomposition of strict support, we may assume $T$ is irreductible and $\cM \in HM_{T}(Z,w).$ Let $\pi: \tilde{T} \rightarrow T$ be a resolution of singularities, and let $\Pi: \tilde{T} \rightarrow Z$ be the induced map. There exists a unique Hodge module $\tcM \in HM_{\tilde{T}}(\tilde{T},w)$ such that $\cM$ is a direct summand of $\cH^{0}(\Pi_{+}\tcM).$  For any $p \in \Z,$ we have
    $$\cH^{i}(Gr^{F}_{p}DR(\Pi_{+}\tcM)) \cong R^{i}\Pi_{*}Gr^{F}_{p}DR(\tcM).$$
    The variety $\tilde{T}$ is smooth and, by the definition of $Gr^{F}_{p}DR(\tcM),$ we have
    $\cH^{i}(Gr^{F}_{p}DR(\tcM)) = 0$ whenever $i < -n_{T}.$ Therefore, we obtain
    $$\cH^{i}(Gr^{F}_{p}DR(\Pi_{+}\tcM)) \cong R^{i}\Pi_{*}Gr^{F}_{p}DR(\tcM) = 0 \quad \text{ for $i < -n_{T}$}.$$
    The lemma now follows because $\cM$ is a direct summand of $\cH^{0}(\Pi_{+}\tcM)$ and Saito's decomposition theorem \cite[Thm. 5.3.1]{saito}, 
    $$\Pi_{+}\tcM \simeq \bigoplus_{i \in \Z} \cH^{i}(\Pi_{+}\tcM)[-i].$$
\end{proof}

\begin{thm}
    Let $X$ be any irreducible variety (not necessarily projective) and $U \subseteq X$ any open subset. If $\cM \in HM_{U}(U,w)$  and $\cM_{X} \in HM_{X}(X,w)$ is the unique extension, then the natural map
    $$\cH^{i}(Gr^{F}_{p}DR(\cM_{X})) \rightarrow \cH^{i}(Gr^{F}_{p}DR(j_{+}\cM))$$
    is an isomorphism for $i< -n_{S}$ and injective for $i = -n_{S}.$
\end{thm}

\begin{proof}
    The problem is local, so we may assume $X$ is quasi-projective. We have the natural triangle
    $$ \begin{tikzcd}
        i_{S+}i^{!}_{S}\cM_{X} \arrow[r] & \cM_{X} \arrow[r] & j_{+}\cM \arrow[r, "+1"] & \hfill. \end{tikzcd} $$
        Note that $\cH^{j}( i_{S+}i^{!}_{S}\cM_{X}) = 0$ for $j < 1.$ Therefore, by the previous lemma, whenever $i \leq -n_{S}$ we must have $\cH^{i}(Gr^{F}_{p}DR( i_{S+}i^{!}_{S}\cM_{X})) = 0.$ The theorem now follows because of the exact triangle above.
\end{proof}

\begin{cor}\label{splitThm}
    Let $X$ be any irreducible variety (not necessarily projective) and assume $U \subseteq X$ is a nonsingular open subset. Let $IC^{H}_{X} \in HM_{X}(X,n)$ denote the unique extension of the trivial Hodge module $\Q^{H}_{U}[n] \in HM_{U}(U,n).$ If $f:Y \rightarrow X$ is a strong log resolution of $S$, then the natural map
    $$\cH^{i+ p -n}(Gr^{F}_{-p}DR(IC^{H}_{X})) \rightarrow R^{i}f_{*}\Omega_{Y}^{p}(\log E)$$
    is an isomoprhism for $i+p<n-n_{S}$ and injective for $i+p= n -n_{S}.$ Furthermore, if $g: Y' \rightarrow X$ is any resolution of singularities, then for $p +i \leq n - n_{S} -1$ the map
    $$ R^{i}g_{*}\Omega^{p}_{Y'} \rightarrow R^{i}f_{*}\Omega^{p}_{Y}(\log E)$$
    is surjective. 
\end{cor}

\begin{proof}
   By the previous theorem, the natural map
   $$\cH^{i}(Gr^{F}_{-p}DR(IC^{H}_{X})) \rightarrow \cH^{i}(Gr^{F}_{-p}DR(j_{+}\Q^{H}_{U}[n]))$$
    is an isomorphism for $i< -n_{S}$ and injective for $i = -n_{S}.$ But, for $0 \leq p \leq n,$ the complex $Gr^{F}_{-p}DR(j_{+}\Q^{H}_{U}[n])$ is quasi-isomorphic to $\bR f_{*} \Omega^{p}_{Y}(\log E)[n-p].$ To finish the proof, the Hodge module $IC^{H}_{X}$ is a direct summand of $g_{+}\Q^{H}_{Y'}[n].$ Hence, we have a decomposition
    $$\bR g_{*}\Omega^{p}_{Y'}[n-p] \simeq Gr^{F}_{-p}DR(IC^{H}_{X}) \oplus M,$$
    with $M \in D^{b}_{coh}(\cO_{X})$. Hence the natural map $R^{i}g_{*}\Omega^{p}_{Y'} \rightarrow \cH^{i +p -n}(Gr^{F}_{-p}DR(IC^{H}_{X}))$ is surjective, and the corollary holds. 
   
\end{proof}

\begin{cor}\label{IC cor 2}
    Let $X$ be an irreducible projective variety and assume $U \subseteq X$ is a nonsingular open subset. Let $IC^{H}_{X} \in HM_{X}(X,n)$ denote the unique extension of the trivial Hodge module $\Q^{H}_{U}[n] \in HM_{U}(U,n).$ Then the natural map
    $$\cH^{i +p -n}(Gr^{F}_{-p}DR(IC^{H}_{X})) \rightarrow R^{i}j_{*} \Omega^{p}_{U}$$
    is an isomorphism for $p +i \leq n -n_{S} -2$ and injective for $p +i = n -n_{S} -1.$ Furthermore, if $g: Y' \rightarrow X$ is a resolution of singularities, then the natural map
    $$ R^{i}g_{*}\Omega^{p}_{Y'} \rightarrow R^{i}j_{*}\Omega^{p}_{U}$$
    is surjective for $p +i \leq n - n_{S} -2.$
\end{cor}

 From the previous result, we can prove a particular case of Flenner's theorem \cite{flenner}.
 
 \begin{cor}\cite{flenner}\label{cor4.8}
 Let $X$ be an irreducile projective variety and $U \subseteq X$ the nonsingular set of $X.$ If $g: Y' \rightarrow X$ is any resolution of singularities, then the natural map
 $$
 g_{*}\Omega^{p}_{Y'} \rightarrow j_{*}\Omega^{p}_{U}
 $$
 is an isomorphism whenever $p \leq n -n_{S} -2$.
 \end{cor}

The previous corollary can be applied to prove a particular case of the Zariski-Lipman conjecture. The Zariski-Lipman conjecture claims that if $Z$ is an algebraic variety and the tangent bundle $T_{Z}$ is locally free, then $Z$ is smooth. Lipman \cite{lipman} proved an algebraic variety $Z$ is normal if the tangent bundle $T_{Z}$ is locally free and proved the conjecture for a few cases. By Corollary \ref{cor4.8}, we can use the proof given in \cite[Sect. 6]{GKKP} to prove the following theorem. 

\begin{thm}\cite{flenner}\label{thm4.9}
Suppose $X$ is an irreducible projective variety and $S$ is the singular locus. If the tangent bundle $T_{X}$ is locally free and $n - n_{S} \geq 3$, then $X$ is smooth.
\end{thm}

From Theorem \ref{thm4.9}, we can deduce Steenbrink and van Straten's result of the Zariski-Lipman conjecture for isolated singularities. 

\begin{cor}\cite{svs}\label{cor4.10}
Assume $Z$ is an irreducible variety and the tangent bundle $T_{Z}$ is locally free. If $Z$ has isolated singularities and $\dim Z > 2$, then $Z$ is smooth.
\end{cor}

\begin{proof}
The problem is local. So, we may assume $Z$ is quasi-projective. Since $Z$ has isolated singularities, we can embed $Z$ into a projective variety $X$ such that the singular locus of $X$ coincides with the singular locus of $Z$. Hence, $X$ is an irreducible projective variety such that the tangent bundle $T_{X}$ is locally free and $\dim X > 2.$ By Theorem \ref{thm4.9}, $X$ is smooth. Therefore, $Z$ is smooth as well.  
\end{proof}

\section{GAGA-Type Theorem for Quasi-Projective Varieties}\label{GAGA-type}
To finish this papper, let $X$ be an irreducible projective variety and $X^{an}$ the associated analytic space of $X$ with the natural map $\Psi: X^{an} \rightarrow X.$ In the settings of Theorem \ref{prop1.9}, recall in the proof we had $\cG:=f\vert_{V \ast}\bR \cH om_{V}(\cF\vert_{V}, \omega^{\bullet}_{V}).$ We showed the sheaf $R^{p}j_{*}\cG^{-l}$ was coherent for $p \leq n- n_{S} -2.$ By \cite[Thm. B]{Siu}, the natural map $(R^{p}j_{*}\cG^{-l})^{an} \rightarrow R^{p}j^{an}_{*}(\cG^{-l})^{an}$ is an isomorphism for $p \leq n - n_{S} -2.$ So consider the following diagram with the natural maps, 
$$\begin{tikzcd}   & (R^{p}j_{*}\sigma_{>-l}\cG)^{an} \arrow[r] \arrow[d] & (R^{p}j_{*}\cG)^{an} \arrow[r] \arrow[d] & (R^{p}j_{*}\cG^{-l}[l])^{an} \arrow[r] \arrow[d] & (R^{p+1}j_{*}\sigma_{>-l}\cG)^{an}  \arrow[d] & \hfill \\
 & R^{p}j^{an}_{*}\sigma_{>-l}\cG^{an} \arrow[r]  & R^{p}j^{an}_{*}\cG^{an} \arrow[r]  & R^{p}j^{an}_{*}(\cG^{-l}[l])^{an} \arrow[r]  & R^{p+1}j^{an}_{*}\sigma_{>-l}\cG^{an}   &  \end{tikzcd}$$
If $l \geq 1$, then by induction on $l,$ the natural map $(R^{p}j_{*}\sigma_{>-l}\cG)^{an}  \rightarrow R^{p}j^{an}_{*}\sigma_{>-l}\cG^{an}$ is an isomorphism for $p \leq -n_{S} -q.$ Since the natural map $(R^{p}j_{*}\cG^{-l}[l])^{an} \rightarrow R^{p}j^{an}_{*}(\cG^{-l})^{an}[l]$ is an isomorphism for $p \leq -n_{S} -q -1,$ the natural map $(R^{p} j_{*}\cG)^{an} \rightarrow R^{p} j^{an}_{*}\cG^{an}$ is an isomorphism for $p \leq -n_{S} - q -1.$

 Since the complex $\cF\vert_{V} \in D^{b}_{coh}(\cO_{V})$ is a complex of locally free sheaves, and because the functor $F \mapsto F^{an}$ is flat, we have the quasi-isomorphism
$$(f\vert_{V \ast}\bR \cH om_{V}(\cF\vert_{V}, \omega^{\bullet}_{V}))^{an} \simeq f^{an}\vert_{V^{an} \ast}\bR \cH om_{V^{an}}(\cF^{an}\vert_{V^{an}}, \omega^{\bullet}_{V^{an}}).$$
Combining these facts, we obtain the following proposition.
\begin{prop}\label{an}
    With the same assumptions as Theorem \ref{prop1.9}, the natural map
    $$R^{p}j_{\ast}\bR \cH om_{U}(f\vert_{V \ast}\cF\vert_{V}, \Omega^{n}_{U}[n])^{an} \rightarrow R^{p}j^{an}_{\ast}\bR \cH om_{V^{an}}(f^{an}\vert_{V^{an} \ast}\cF^{an}\vert_{V^{an}}, \Omega^{n}_{U^{an}}[n])$$
    is an isomorphism whenever $p \leq -n_{S} -q -1$.
\end{prop}

Combining Proposition \ref{an},  Proposition \ref{prop4.1}, and using a local version of GAGA \cite{SGA1}[XII Thm. 4.2], we have the following result for mixed Hodge modules on the associated analytic space of $X.$

\begin{thm}\label{analyticThm.}
    Assume $f: Y \rightarrow X$ is a strong log resolution of $S$. If $\cM \in HM_{V}(V,w)$ is a variation of Hodge structure, then the natural maps
    $$ (R^{i}j_{\ast}f\vert_{V \ast}Gr^{F}_{-p }DR(\cM))^{an} \rightarrow R^{i}j^{an}_{\ast}f^{an}\vert_{V^{an} \ast}Gr^{F}_{-p }DR(\cM^{an})$$
    $$R^{i}f^{an}_{\ast}Gr^{F}_{-p}DR(\cM^{an}(\ast E)) \rightarrow R^{i}j^{an}_{\ast}f^{an}\vert_{V^{an} \ast}Gr^{F}_{-p }DR(\cM^{an})$$
    
are isomorphisms for $i \leq -n_{S} -2.$
\end{thm}

 Again,  assume $f: Y \rightarrow X$ is a strong log resolution of $S$. Since $f\vert_{V}: V \rightarrow U$ is an isomorphism, any mixed Hodge module on $V$ may be regarded as a mixed Hodge module on $U.$ If $\cM \in HM_{U}(U,w)$ is a variation of Hodge structure, then there are spectral sequences
$$E^{a,b}_{2}=H^{a}(X, R^{b}j_{*}Gr^{F}_{-p}DR(\cM)) \Rightarrow H^{a+b}(U, Gr^{F}_{-p}DR(\cM))$$
$$^{'}E^{a,b}_{2}=H^{a}(X^{an}, R^{b}j^{an}_{*}Gr^{F}_{-p}DR(\cM^{an})) \Rightarrow H^{a+b}(U^{an}, Gr^{F}_{-p}DR(\cM^{an})).$$
By \cite{gaga}, for $a \in \mathbb{Z}$ and $b \leq -n_{S} -2$ there are isomorphisms 
$$H^{a}(X, R^{b}j_{*}Gr^{F}_{-p}DR(\cM)) \cong H^{a}(X^{an}, (R^{b}j_{*}Gr^{F}_{-p}DR(\cM))^{an}) \cong H^{a}(X^{an}, R^{b}j^{an}_{*}Gr^{F}_{-p}DR(\cM^{an})).$$
Hence, the two spectral sequences above give us 
$$H^{i}(U,Gr^{F}_{-p}DR(\cM)) \cong H^{i}(U^{an},Gr^{F}_{-p}DR(\cM^{an})) \quad \text{if $i \leq -n_{S} -2$}.$$
Therefore, by Theorem \ref{analyticThm.}, most of the results from this paper hold for Hodge modules on algebraic manifolds.

\printbibliography[
heading=bibintoc,
title={References}
] 

\end{document}